\documentclass[journal]{IEEEtran}  % Comment this line out
\usepackage{amsmath}
\usepackage{amssymb}
\usepackage{amsthm}
\usepackage{enumitem}
\usepackage{multicol}
\usepackage{graphics} % for pdf, bitmapped graphics files
\usepackage{epsfig} % for postscript graphics files
\usepackage{epstopdf}
\usepackage{mathptmx} % assumes new font selection scheme installed
\usepackage{times} % assumes new font selection scheme installed
\usepackage{amsmath} % assumes amsmath package installed
\usepackage{amssymb}  % assumes amsmath package installed
\usepackage{subfigure}% in preamble
\usepackage[utf8]{inputenc}
\usepackage[english]{babel}
\usepackage{fancyhdr}
\usepackage{dsfont}

\usepackage{color}

\newtheorem{thm}{Theorem}
\newtheorem{defn}{Definition}
\newtheorem{prop}{Proposition}
\newtheorem{lem}{Lemma}

\newtheorem{cor}{Corollary}

\newtheorem{ex}{Example}

\newcommand{\mcl}[1]{\mathcal{#1}}

\DeclareMathOperator*{\esssup}{ess\,sup}

\newcommand{\R}{\mathbb{R}}

\newcommand{\N}{\mathbb{N}}

\newcommand{\eps}{\varepsilon}

\title{\LARGE \bf
Converse Lyapunov Functions
and Converging Inner Approximations
to Maximal Regions of Attraction
of Nonlinear Systems
}

\author{Morgan Jones%
	\thanks{M. Jones is with the School for the Engineering of Matter, Transport and Energy, Arizona State University, Tempe, AZ, 85298 USA. e-mail: {\tt \small morgan.c.jones@asu.edu } },
	Matthew M. Peet% <-this % stops a space
	\thanks{M. Peet is with the School for the Engineering of Matter, Transport and Energy, Arizona State University, Tempe, AZ, 85298 USA. e-mail: {\tt \small mpeet@asu.edu } }
}

\begin{document}

\maketitle
\thispagestyle{plain}
\pagestyle{plain}

\begin{abstract}
	This paper considers the problem of approximating the ``maximal'' region of attraction (the set that contains all asymptotically stable sets) of any given set of locally exponentially stable nonlinear Ordinary Differential Equations (ODEs) with a sufficiently smooth vector field. Given a locally exponential stable ODE with a differentiable vector field, we show that there exists a globally Lipschitz continuous converse Lyapunov function whose $1$-sublevel set is equal to the maximal region of attraction of the ODE.
We then propose a sequence of $d$-degree Sum-of-Squares (SOS) programming problems that yields a sequence of polynomials that converges to our proposed converse Lyapunov function uniformly from above in the $L^1$ norm.
We show that each member of the sequence of $1$-sublevel sets of the polynomial solutions to our proposed sequence of SOS programming problems are certifiably contained inside the maximal region of attraction of the ODE, and moreover, we show that this sequence of sublevel sets converges to the maximal region of attraction of the ODE with respect to the volume metric.
We provide numerical examples of estimations of the maximal region of attraction for the Van der Pol oscillator and a three dimensional servomechanism.
\end{abstract}
\vspace{-0.6cm}
\section{Introduction}
For a given equilibrium point, a Region of Attraction (ROA) of a nonlinear Ordinary Differential Equation (ODE) is defined as a set of initial conditions for which the solution map of the ODE tends to that equilibrium point.
The \textit{maximal} ROA of an equilibrium point, meanwhile, is defined as the ROA which contains all other ROAs of that equilibrium point.
Specifically, for an ODE $\dot{x}(t)=f(x(t))$, we denote the solution map (known to exist when $f$ is Lipschitz continuous) of the ODE by $\phi_f: \R^n \times \R \to \R^n$ which satisfies
\vspace{-0.2cm}
\begin{align*}
\frac{d}{dt} \phi_f(x,t) & = f(\phi_f(x,t)) \text{ for all } x \in \R^n \text{ and } t \ge 0, \\
\phi_f(x,0) & =x \text{ for all } x \in \R^n,
\end{align*}
where $f:\R^n \to \R^n$ is such that $f(0)=0$. The maximal ROA is then defined as
\vspace{-0.2cm}
\begin{align*}
ROA_f:=\{x \in \R^n: \lim_{t \to \infty} ||\phi_f(x,t)||_2=0\}.
\end{align*}
%Given a locally exponentially stable ODE, the contribution of this paper is to propose a family Sum-of-squares (SOS) optimization problems that yield a sequence of sublevel sets that converge to $ROA)f$ with respect to the volume metric.

The problem of computing sets which accurately approximate the maximal ROA with respect to some set metric plays a central role in the stability analysis of many engineering applications. For instance, knowledge of the ROA provides a metric for the susceptibility of the F/A-18 Hornet aircraft experiencing an unsafe out of control flight departure phenomena, called falling leaf mode~\cite{Chakraborty2011335,chakraborty2011susceptibility}.

If the matrix $A \in \R^{n \times n}$ is Hurwitz (the real part of the eigenvalues of $A$ are all negative) then the associated linear system, with vector field $f(x)=Ax$, has a maximal ROA that can be found exactly as $ROA_f=\R^n$. In the more general case of nonlinear systems there is no known general analytical formula for $ROA_f$. However, for particular nonlinear systems, such as those arising from gradient flow dynamics, the maximal ROA can be expressed analytically~\cite{mohammadi2018stability}. In the absence of an analytical formula for $ROA_f$ in recent years there has been considerable interest in discovering numerical methods for approximating $ROA_f$ rather than finding $ROA_f$ exactly.

Lyapunov`s second method is arguably the most widely used technique for finding ROAs associated with an ODE~\cite{kellett2015classical}. Rather than solving the ODE directly to find a closed form expression of the solution map, ROAs can be computed indirectly by searching for a ``generalized energy function", called  a Lyapunov function. A Lyapunov function of an ODE is any function that is positive everywhere, apart from the origin where it is zero, and is strictly decreasing along the solution map of the ODE. Specifically, if we can find a function $V$ such that $V(0)=0$ and $V(x)> 0$ for all $ x \ne 0$, then if $\nabla{V}(x)^Tf(x)$ is negative over the sublevel set $\{x \in \R^n :V(x) \le a \}$ we have that $\{x \in \R^n :V(x) \le a \} \subseteq ROA_f$ is a ROA~\cite{Hahn19671}. For linear systems, $f(x)=Ax$ where $A \in \R^{n \times n}$, a necessary and sufficient condition for $ROA_f=\R^n$ is that there exists a quadratic Lyapunov function of form $V(x)=x^T P x$ where $P>0$. Thus, in this case, the problem of finding the maximal ROA of a linear system is reduced to solving the Linear Matrix Inequality (LMI) $A^TP+PA<0$ for $P>0$.

In the case of nonlinear systems a common approach for finding Lyapunov functions has been to generalize the search from quadratic functions, $V(x)=x^T P x$, to Sum-of-Square (SOS) polynomials functions, $V(x)=Z_d(x)^T P Z_d(x)$ where $Z_d$ is the degree $d \in \N$ monomial vector. Then, to find a Lyapunov function we must solve an SOS optimization problem, rather than solving an LMI (as was the case for linear systems). Over the years, many SOS optimization problems have been proposed for ROA estimation~\cite{4471858,zheng2018computing,anderson2015advances,colbert2018using}. Recently in~\cite{cunis2020sum}, SOS was used to estimate the region of attraction of an uncrewed aircraft; in~\cite{valmorbida2017region} an SOS based algorithm was proposed to construct a rational Lyapunov function that yields an estimate of the ROA; in~\cite{awrejcewicz2021estimating} a recursive procedure for constructing the polynomial Lyapunov functions was proposed.

%in~\cite{4471858} where SOS was used to construct Lyapunuov functions comprising of point-wise maximums of polynomial functions; in~\cite{valmorbida2017region} where SOS was used to construct rational Lyapunov functions; and in~\cite{zheng2018computing} multiple Lyapunuov function like functions were constructed using SOS for region of attraction estimation of switched hybrid systems.

Despite the recent success of modern attempts to find accurate approximations of the maximal ROA, to the best of our knowledge, a numerical algorithm that can be proven to provide an approximation of the maximal ROA arbitrarily well with respect to any set metric has yet to be proposed. Many of the current numerical methods for finding ROAs use SOS programming to find polynomial Lyapunov functions.
However, barring any assumptions on the existence of a sufficiently smooth Lyapunov function, it is currently unknown how well polynomial functions can approximate the maximal ROA of a given nonlinear ODE. Concerningly, several counter examples~\cite{ahmadi2011globally,ahmadi2018globally} show that there exist globally asymptotically stable systems ($ROA_f= \R^n$) with polynomial vector fields, but for which there does not exist any associated polynomial Lyapunov function that can certify global asymptotic stability (not even locally in the case of~\cite{ahmadi2018globally}). On the other hand, for systems that are locally exponentially stable it has been shown in~\cite{peet2009exponentially} that there always exists a polynomial Lyapunov function that can certify local exponential stability. This result has been extended in~\cite{leth2017existence} to show that there always exist polynomial Lyapunov functions that can certify a system is locally rationally stable (a weaker form of stability than exponential stability) under the assumption that there exists a smooth Lyapunov function (that need not be polynomial). Furthermore, for systems with homogeneous vector fields it has been shown in~\cite{ahmadi2019algebraic} that there always exists a rational Lyapunov function that is the solution to some SOS problem.

For work that is concerned with using SOS to approximate the maximal ROA of locally exponentially stable systems we mention~\cite{Jones_ROA}. It was shown in~\cite{Jones_ROA} that under the assumption that there exists a sufficiently smooth Lyapunov function, there exists a polynomial Lyapunov function that yields a sublevel set that approximates $ROA_f$ arbitrarily well with respect to the Hausdorff metric. We note that the conservatism of the assumption that there exists a sufficiently smooth Lyapunov function is currently unknown. Moreover, the proposed algorithm for approximating the maximal ROA found in~\cite{Jones_ROA} is only conjectured to yield an arbitrarily close approximation of the maximal ROA but has yet to be proven.

The goal of this paper is to design an algorithm that approximates the maximal ROA of a given locally exponentially stable ODE arbitrarily well. In order to achieve this goal we propose a new converse Lyapunov function (given in Eqn.~\eqref{fun: ZLF}) whose $1$-sublevel set is equal to $ROA_f$. Our proposed converse Lyapunov function is shown to be sufficiently smooth - meaning it can be approximated by a polynomial. After proposing such a converse Lyapunov function, we are then able to design a sequence of SOS Optimization Problems~\eqref{opt: SOS for ROA} and prove that this sequence yields a sequence of polynomials that converges to our proposed converse Lyapunov function uniformly from above in the $L^1$ norm. Finally, we show that since this sequence of polynomials converges to our proposed converse Lyapunov function in the $L^1$ norm from above, their associated sequence of $1$-sublevel sets must also converge in the volume metric to the $1$-sublevel set of our proposed converse Lyapunov function (which is equal to the maximal region of attraction of the ODE). Therefore, for a given locally exponentially stable ODE, the goal of this paper is: 1) To establish the existence of a globally Lipschitz continuous converse Lyapunov function whose $1$-sublevel set is equal to $ROA_f$. 2) To propose the first numerical algorithm that can approximate the maximal ROA arbitrarily well with respect to some set metric. Furthermore, our numerical algorithm yields an inner approximation of $ROA_f$ (that is solution maps initialized inside our approximation of $ROA_f$ asymptotically coverage to the origin); a useful property for the safety analysis of dynamical systems.

%The contributions of this paper are related. Once the existence of a globally Lipschitz continuous converse Lyapunov function is shown, we are able to design a numerical algorithm for the estimation of the region of attraction by proposing a technique for approximating our converse Lyapunov function. Specifically, we propose a family of $d$-degree SOS Optimization Problems~\eqref{opt: SOS for ROA} that yields a sequence of polynomials that converge to our proposed converse Lyapunov function under the $L^1$ norm as $d \to \infty$.  Then, applying Corollary~\ref{cor: close in L1 implies close in V norm}, it follows that the sublevel sets of the solutions to our proposed family SOS problem converge to the sublevel set of our proposed converse Lyapunov function; thus yielding convergent estimations of the region of attraction. The contributions of this paper imply that polynomials can be used to estimate the region of attraction for a large class of ODEs, thus providing a theoretical basis for alternative existing SOS based algorithms that estimate the region of attraction.

The rest of this paper is organized as follows. In Section~\ref{sec: ROA solution maps} we define the maximal region of attraction of an ODE in terms of the solution map of the ODE. In Section~\ref{sec: estimate ROA} we formulate the problem of approximating the region of attraction as an optimization problem. In Section~\ref{sec: A globally lip cts converse LF} we propose a globally Lipschitz continuous Lyapunov function that characterizes the maximal region of attraction. In Section~\ref{sec: bounding LF from above} we propose a convex optimization problem for the approximation of our proposed converse Lyapunov function in the $L^1$-norm. In Section~\ref{sec: SOS opt} we tighten this optimization problem to an SOS programming problem. Finally, several numerical examples are given in Section~\ref{sec: numerical examples} and our conclusion is given in Section~\ref{sec: conclusion}.

\vspace{-0.3cm}
\section{Notation} \label{sec: notation}
%We now introduce notation used throughout the paper.
\vspace{-0.3cm}
\subsection{Set Notation}
We denote the power set of $\R^n$, the set of all subsets of $\R^n$, as $P(\R^n)=\{X:X\subset \R^n \}$. For two sets $A,B \in \R^n$ we denote $A/B= \{x \in A: x \notin B\}$. For $x \in \R^n$ we denote $||x||_p= \left( \sum_{i =1}^n x_i^p \right)^{\frac{1}{p}}$. For $\eta>0$ and $y \in \R^n$ we denote the set $B_\eta(y)= \{x\in \R^n : ||x-y||_2< \eta\}$. For a set $X \subset \R^n$ we say $x \in X$ is an interior point of $X$ if there exists $\eps>0$ such that $\{y \in \R^n: ||x-y||< \eps\}\subset X$. We denote the set of all interior points of $X$ by $X^\circ$. The point $x \in X$ is a limit point of $X$ if for all $ \eps>0$ there exists $ z \in \{y \in \R^n / \{x\}: ||x-y||<\eps\}$ such that $z \in X$; we denote the set of all limit points of $X$, called the closure of $X$, as $(X)^{cl}$. Moreover, we denote the boundary of $X$ by $\partial X= (X)^{cl} / X^\circ$. For $A \subset \R^n$ we denote the indicator function by $\mathds{1}_A : \R^n \to \R$ that is defined as $\mathds{1}_A(x) = \begin{cases}
& 1 \text{ if } x \in A\\
& 0 \text{ otherwise.}
\end{cases}$ For $B \subseteq \R^n$,  $\mu(B):=\int_{\R^n} \mathds{1}_B (x) dx$ is the Lebesgue measure of $B$. Let us denote bounded subsets of $\R^n$ by $\mcl B:=\{B \subset \R^n: \mu(B)<\infty\}$.  If $M$ is a subspace of a vector space $X$ we denote equivalence relation $\sim_M$ for $x,y \in X$ by $x \sim_M y$ if $x-y \in M$. We denote quotient space by $X \pmod M:=\{ \{y \in X: y \sim_M x \}: x \in X\}$. For an open set $\Omega \subset \R^n$ and $\sigma>0$ we denote $<\Omega>_\sigma:=\{x \in \Omega: B(x,\sigma) \subset \Omega \}$.
\vspace{-0.5cm}
\subsection{Continuity Notation}
 Let $C(\Omega, \Theta)$ be the set of continuous functions with domain $\Omega \subset \R^n$ and image $\Theta \subset \R^m$. We denote the set of locally and uniformly Lipschitz continuous functions on $\Theta_1 \text{ and }\Theta_2$, Defn.~\ref{defn:lip cts}, by $LocLip(\Theta_1,\Theta_2)$ and $Lip(\Theta_1,\Theta_2)$ respectively. For $\alpha \in \N^n$ we denote the partial derivative $D^\alpha f(x):= \Pi_{i=1}^{n} \frac{\partial^{\alpha_i} f}{\partial x_i^{\alpha_i}} (x)$ where by convention if $\alpha=[0,..,0]^T$ we denote $D^\alpha f(x):=f(x)$. We denote the set of $i$'th continuously differentiable functions by $C^i(\Omega,\Theta):=\{f \in C(\Omega,\Theta): D^\alpha f \in C(\Omega, \Theta) \text{ } \text{ for all } \alpha \in \N^n \text{ such that } \sum_{j=1}^{n} \alpha_j \le i\}$. For $V \in C^1(\R^n, \R)$ we denote $\nabla V:= (\frac{\partial V}{\partial x_1},....,\frac{\partial V}{\partial x_n})^T$. We denote the essential supremum by $\esssup_{x \in X}f(x):=\inf\{a \in \R: \mu(\{x \in X: f(x)>a\})= 0\}$.
 \vspace{-0.5cm}
\subsection{Sobolev Space Notation}
For an open set $\Omega \subset \R^n$ and $p \in [1,\infty)$ we denote the set of $p$-integrable functions by $L^p(\Omega,\R):=\{f:\Omega \to \R \text{ measurable }: \int_{\Omega}|f|^p<\infty    \}$, in the case $p = \infty$ we denote $L^\infty(\Omega, \R):=\{f:\Omega \to \R \text{ measurable }: \esssup_{x \in \Omega}|f(x)| < \infty \}$. For $k \in \N$ and $1 \le p \le \infty$ we denote the Sobolev space of functions with weak derivatives (Defn.~\ref{def: weak deriv}) by $W^{k,p}(\Omega,\R):=\{u\in L^p(\Omega,\R): D^\alpha u \in L^p(\Omega,\R) \text{ for all } |\alpha| \le k \}$. For $u \in W^{k,p}(\Omega,\R)$ we denote the Sobolev norm $||u||_{W^{k,p}(\Omega, \R)}:= \begin{cases}
\left( \sum_{|\alpha| \le k} \int_\Omega (D^\alpha u(x))^p dx \right)^{\frac{1}{p}} \text{ if } 1 \le p < \infty\\
\sum_{|\alpha| \le k} \esssup_{ x \in \Omega  } \{|D^\alpha u(x) |\} \text{ if } p= \infty.
\end{cases}$ In the case $k=0$ we have $W^{0,p}(\Omega,\R)=L^p(\Omega,\R)$ and thus we use the notation $|| \cdot ||_{L^p(\Omega,\R)} :=|| \cdot ||_{W^{0,p}(\Omega,\R)} $. The $\sigma$-mollification of a function $V \in L^1(\Omega, \R)$ is denoted by $[V]_{\sigma}:<\Omega>_\sigma \to \R$ and defined in Eqn.~\eqref{eqn: mollification}.
\vspace{-0.5cm}
\subsection{Polynomial Notation}
We denote the space of polynomials $p: \Omega \to \Theta$ by $\mcl P(\Omega,\Theta)$ and polynomials with degree at most $d \in \N$ by $\mcl{P}_d(\Omega,\Theta)$. We say $p \in \mcl{P}_{2d}(\R^n,\R)$ is Sum-of-Squares (SOS) if for $k \in \{1,...k\}\subset \N$ there exists $p_i \in \mcl{P}_{d}(\R^n,\R)$ such that $p(x) = \sum_{i=1}^{k} (p_i(x))^2$. We denote $\sum_{SOS}^d$ to be the set of SOS polynomials of at most degree $d \in \N$ and the set of all SOS polynomials as $\sum_{SOS}$. We denote $Z_d: \R^n \times \R \to \R^{\mcl N_d}$ as the vector of monomials of degree $d \in \N$ or less, where $\mcl N_d:= {d+n \choose d}$.

\section{Regions of Attraction are Defined Using Solution Maps of Nonlinear ODEs} \label{sec: ROA solution maps}
Consider a nonlinear Ordinary Differential Equation (ODE) of the form
\begin{equation} \label{eqn: ODE}
\dot{x}(t) = f(x(t)), \quad x(0)=x_0\in \R^n, \quad t \in [0,\infty),
\end{equation}
where $f: \R^n \to \R^n$ is the vector field and $x_0 \in \R^n$ is the initial condition. Note that, throughout this paper we will assume $f(0)=0$ so the origin is an equilibrium point.

Given $D \subset \R^n$, $I \subset [0, \infty)$, and an ODE~\eqref{eqn: ODE} we say any function $\phi_f :D \times I \to \R^n$ satisfying
\begin{align} \label{soln map property}
&\frac{\partial \phi_f(x,t)}{\partial t}= f(\phi_f(x,t)) \text{ for } (x,t) \in D \times I,\\ \nonumber
& \phi_f(x,0)=x \text{ for } x \in D,\\ \nonumber
& \phi_f(\phi_f(x,t),s)= \phi_f(x,t+s) \text{ for } x \in D \text{ } t,s \in I \text{ with } t+s \in I,
\end{align}
is a solution map of the ODE~\eqref{eqn: ODE} over $D \times I$. For simplicity throughout the paper we will assume there exists a unique solution map to the ODE~\eqref{eqn: ODE} over all $(x,t) \in \R^n \times [0,\infty)$ (uniqueness and existence of a solution map sufficient for the purposes of this paper, such as for initial conditions inside some invariant set, like the Region of Attraction~\eqref{eqn: ROA}, and for all $t \ge 0$, can be shown to hold under minor smoothness assumption on $f$, see~\cite{Khalil_1996}).

We now use the solution map of the ODE~\eqref{eqn: ODE} to define notions of stability.
\begin{defn} \label{defn: asym and exp stab}
	We say the set $U \subset \R^n$ is an asymptotically stable set of the ODE~\eqref{eqn: ODE} if:
	\begin{enumerate}
		\item $U$ contains a neighborhood of the origin.
		\item For any $x \in U$ we have that $\phi_f(x,t) \in U$ for all $t \in [0,\infty)$ and $\lim_{t \to \infty} \phi_f(x,t)=0$.
	\end{enumerate}
	Furthermore, if there also exists $\delta, \mu>0$ such that for any $x \in U$ we have that
	\begin{align} \label{eqn: exp stab}
	||\phi_f(x,t)||_2 \le \mu e^{-\delta t} ||x||_2 \text{ for all } t \ge 0,
	\end{align}
	then we say $U\subset \R^n$ is an exponentially stable set of the ODE~\eqref{eqn: ODE}.
\end{defn}
%The Region of Attraction of the ODE~\eqref{eqn: ODE} is defined as follows.
\begin{defn} \label{defn: ROA}
	The (Maximal) Region of Attraction (ROA) of the ODE~\eqref{eqn: ODE} is defined as the following set:
	\begin{align} \label{eqn: ROA}
	ROA_f:= \{x \in \R^n: \lim_{t \to \infty} ||\phi_f(x,t)||_2=0\}.
	\end{align}
\end{defn}
The ROA of the ODE~\eqref{eqn: ODE} can be thought of as the ``maximal" asymptotically stable set. That is if $U \subset \R^n$ is an asymptotically stable set of the ODE~\eqref{eqn: ODE} then $U \subseteq ROA_f$. Moreover, as we will show next, the ROA is an open set.
\begin{lem}[Lemma 8.1 \cite{Khalil_1996} ] \label{lem: ROA open}
	Consider an ODE of Form~\eqref{eqn: ODE}. The set $ROA_f$ (Defined in Eqn.~\eqref{eqn: ROA}) is open.
\end{lem}

Before proceeding we introduce some useful notation for the $\eta$-ball set entry times of solution maps. For a given function $\phi_f:\R^n \times \R \to \R^n$, $x \in ROA_f$, and $\eta>0$ we denote
\begin{align} \label{fun: set entry times}
F_\eta(x):=\inf\{t \ge 0: \phi_f(x,t) \in B_\eta(0) \}.
\end{align}
%\subsection{Properties of Solution Maps}
We now state two important properties of solution maps used in many of the proofs presented in this paper.
\begin{lem}[Exponential divergence of solution maps. Page 392 \cite{Hirch_2004}] \label{lem: exp divergence of solution maps}
	Suppose $f \in {C^2(\R^n, \R)}$ and there exists $\theta,R>0$ such that $||D^\alpha f(x)||_2<\theta$ for all $x\in B_R(0)$ and any $||\alpha||_1 \le 2$, where $\alpha \in \N^n$. Then the solution map satisfies the following inequality
	\vspace{-0.15cm}
	\begin{equation}
	||\phi_f(x,t) - \phi_f(y,t)||_2 \le e^{ \theta t} ||x-y||_2  \text{ for } t \ge 0 \text{ and } x, y \in ROA_f.
	\end{equation}
\end{lem}

\begin{lem}[Smoothness of the solution map. Page 149 \cite{Hirch_2004}]
	\label{lem: diff soln map}
	If $f \in C^1(\R^n,\R^n)$ then the solution map is such that $\phi_f \in C^1(\R^n \times \R, \R)$. %{The solution map is continuous for all initial conditions and time.}
\end{lem}

\section{The Problem of Approximating The ROA} \label{sec: estimate ROA}
Consider $f \in C^2(\R^n,\R^n)$. The goal of this paper is to compute an optimal (with respect to some set metric) inner approximation of $ROA_f$ (given in Defn.~\ref{defn: ROA}). That is, we would like to solve the following optimization problem:
\vspace{-0.15cm}
\begin{align} \label{opt: geometric ROA}
& \min_{X \in \mcl C}\{D(ROA_f,X)\}\\ \nonumber
& \text{ such that } X \subseteq ROA_f,
\end{align}
where $\mcl C \subset P(\R^n)$ is some constraint set (recalling from Sec.~\ref{sec: notation} that $P(\R^n)$ is the power set of $\R^n$) and $D: \{Y: Y\subset \R^n \} \times \{Y: Y\subset \R^n \} \to \R$ is some set metric. Note, if the constraint set contains all subsets of $\R^n$, that is $\mcl C =P(\R^n)$, then trivially the optimization problem is solved by the region of attraction, $X=ROA_f$.

The optimization problem given in Eqn.~\eqref{opt: geometric ROA} is fundamentally ``geometric in nature" since it is solved by finding a subset of Euclidean space, $X \subset \R^n$. In this paper we reformulate the optimization problem given in Eqn.~\eqref{opt: geometric ROA} as an optimization problem that is ``algebraic in nature", being solved by a function rather than a set. In order to formulate such an ``algebriac" optimization problem we first propose a converse Lyapunov function (given later in Eqn.~\eqref{fun: ZLF}), denoted here as $W$, whose $1$-sublevel set is equal to $ROA_f$; that is $ROA_f= \{x \in \R^n: W(x)<1\}$. Then rather than finding the set ``closest" to $ROA_f$, we find the ``closest" $d$-degree polynomial to $W$ with respect to the $L^1$ norm. Thus we consider the following ``algebriac" problem:
\vspace{-0.2cm}
\begin{align} \label{opt: algebriac ROA}
& P_d\in \arg \min_{J \in \mcl P_d(\R^n,\R)} \int_{\Lambda} |J(x)-W(x)| dx\\ \nonumber
& \text{ such that } W(x) \le J(x) \text{ for all } x \in \Omega,
\end{align}
where $ROA_f \subseteq \Lambda \subseteq \Omega \subset \R^n$. Then, Cor.~\ref{cor: close in L1 implies close in V norm} (found in Appendix~\ref{sec: appendix 2}) can be used to show that $\{x \in \Lambda: P_d(x) < 1 \}$ converges to $\{x \in \Lambda: W(x) < 1 \}=ROA_f$ as $d \to \infty$ with respect to the volume metric (given in Eqn.~\eqref{eqn: volume metric}).

Solving the optimization problem given in Eqn.~\eqref{opt: algebriac ROA} has the following challenges:
\begin{enumerate}
	\item Does there exist a converse Lyapunov function $W:\R^n \to \R$ such that $ROA_f= \{x \in \R^n: W(x)<1\}$?
	\item Can the constraint, $W(x) \le J(x) \text{ for all } x \in \Omega$, be tightened to a convex constraint without necessarily having an analytical formula for $W$?
	\item Does the solution, $P_d$, tend towards $W$ with respect to the $L^1$ norm as $d \to \infty$?
\end{enumerate}
In the next section we tackle the first of these challenges. We propose a converse Lyapunov function, $W$, whose $1$-sublevel set is equal to $ROA_f$. Then, in Sec.~\ref{sec: bounding LF from above} we tackle the second challenge; we propose a sufficient condition, in the form of a linear partial differential inequality, that when satisfied by a function $J$ implies $W(x) \le J(x) \text{ for all } x \in \Omega$. Finally, in Appendix~\ref{sec: appendix mollification}, we tackle the third challenge of showing that there exists a sequence of $d$-degree polynomials, feasible to Opt.~\eqref{opt: algebriac ROA} for $d \in \N$, that converges to $W$ with respect to the $L^1$ norm. For implementation purposes Opt.~\eqref{opt: algebriac ROA} is then tightened to an SOS optimization problem, given in Eqn.~\eqref{opt: SOS for ROA}, that can be efficiently numerically solved. The main result of the paper is then given in Theorem~\ref{thm: SOS converges to ROA in V}, showing that our proposed family of $d$-degree SOS Optimization Problems~\eqref{opt: SOS for ROA} yields a sequence of sets that converge to the region of attraction of a given locally exponentially stable ODE with respect to the volume metric as $d \to \infty$.
\section{A Globally Lipschitz Continuous Converse Lyapunov Function That Characterizes the ROA}  \label{sec: A globally lip cts converse LF}
In~\cite{vannelli1985maximal} a converse Lyapunov function, called the maximal Lyapunov function, was proposed. It was shown that for any given asymptotically stable ODE there exists a maximal Lyapunov function whose $\infty$-sublevel set is equal to the region of attraction of the ODE. However, since by definition any maximal Lyapunov function is unbounded outside of the region of attraction it cannot be approximated arbitrarily well (with respect to any norm) by a polynomial over any compact set that contains points outside of the region of attraction (since polynomials are bounded over compact sets). Thus, it is not possible to design an SOS based algorithm that can approximate maximal Lyapunov functions arbitrarily well. To overcome this challenge we propose a new converse Lyapunov function (found in Eqn.~\eqref{fun: ZLF}) whose $1$-sublevel set is equal to $ROA_f$, is globally bounded, and is globally Lipschitz continuous. Before introducing our new converse Lyapunov function let us recall the definition of Lipschitz continuity.
\begin{defn} \label{defn:lip cts}
	Consider sets $\Theta_1 \subset \R^n$ and $\Theta_2 \subset \R^m$. We say the function $F: \Theta_1 \to \Theta_2$ is \textbf{locally Lipschitz continuous} on $\Theta_1 \text{ and }\Theta_2$, denoted $F \in LocLip(\Theta_1, \Theta_2)$, if for every compact set $X \subseteq \Theta_1$ there exists $K>0$ (that may depend on $X$) such that for all $x,y \in X$
	\begin{align} \label{lip def}
	||F(x) - F(y)||_2 \le K ||x - y||_2.
	\end{align}
	If there exists a single $K>0$ such that Eqn.~\eqref{lip def} holds for all $x,y \in \Theta_1$ we say $F$ is \textbf{globally Lipschitz continuous}, denoted $F \in Lip(\Theta_1,\Theta_2)$.
	%We refer to $K>0$ as the Lipschitz constant of $F$ and throughout the paper denote the Lipschitz constant of a Lipschitz continuous function $F$ by $L_F>0$.
\end{defn}

We consider two different types of converse Lyapunov functions. The first converse Lyapunov function (given in Eqn.~\eqref{fun: MLF}) is a special case of those first found in~\cite{massera1949liapounoff} that have the form $V_1(x):=\int_0^\infty G(||\phi_f(x,t)||_2)dt$ for some class K function, $G:[0, \infty) \to [0, \infty)$ (class K is the class of functions which monotonically approach zero at zero). In~\cite{vannelli1985maximal} it was shown that for a locally stable ODE, the $\infty$-sublevel set of $V_1$ is equal to the region of attraction of the ODE; this Lyapunov function was named the maximal Lyapunov function. In this paper we only consider locally exponentially stable systems and hence may restrict ourselves to the special case when $G(y)=y^{2 \beta}$ for some $\beta \in \N$.

The second converse Lyapunov function we consider (found in Eqn.~\eqref{fun: ZLF}) can be thought of as a nonlinear transformation of the first converse Lyapunov function. A function of a similar structure was previously considered in~\cite{zubov1964methods} and took the form $V_2(x):=\exp\left(-\int_0^\infty G(||\phi_f(x,t)||_2)dt\right)-1$. Although~\cite{zubov1964methods} used $V_2$ to certify the stability of a system, $V_2$  is not a Lyapunov function in the classical sense since it is not positive everywhere (unlike our proposed converse Lyapunov function in Eqn.~\eqref{fun: ZLF}). We note that~\cite{zubov1964methods} did establish the globally continuity of $V_2$ but did not show the stronger result that $V_2$ is Lipschitz continuous.% (a property our proposed converse Lyapunov function in Eqn.~\eqref{fun: MLF} is shown to have and integral to the proof of the main result of this paper, Theorem~\ref{thm: SOS converges to ROA in V}) .

Now, consider $f \in LocLip(\R^n, \R^n)$, $\lambda>0$ and $\beta \in \N$. Let us denote the functions $V_{\beta}:ROA_f \to \R$ and $W_{\lambda, \beta}:\R^n \to \R$ where
\vspace{-0.2cm}
\begin{align}  \label{fun: MLF}
&V_{\beta}(x):=\int_{0}^\infty ||\phi_f(x,t)||_2^{2 \beta} dt,\\ \label{fun: ZLF}
&W_{\lambda, \beta}(x):=\begin{cases}
1 - \exp(- \lambda \int_{0}^\infty ||\phi_f(x,t)||_2^{2 \beta} dt) \text{ when } x \in ROA_f\\
1 \text{ otherwise.}
\end{cases}
\end{align}
\vspace{-0.8cm}
\subsection{Converse Lyapunov Functions that Characterize the ROA}
The function, $V_\beta$, given in Eqn.~\eqref{fun: MLF} is a special case of a class of Lyapunov functions called maximal Lyapunov functions~\cite{vannelli1985maximal}. In the following lemma we will show that $V_\beta$ tends to infinity for sequences of points approaching the boundary of the region of attraction and is finite inside the region of attraction. %Hence, the maximal Lyapunov function can used to approximate $ROA_f$ exactly; since the region of attraction is given by the set of points where $V_\beta$ is finite.

\begin{lem} \label{lem: MLF}
	Consider $f \in LocLip(\R^n, \R)$, $\beta \in \N$ and $V_\beta$ given in Eqn.~\eqref{fun: MLF}. Suppose there exists $R,\eta>0$ such that $ROA_f \subset B_R(0)$ and $B_\eta(0)$ is an exponentially stable set (Defn.~\ref{defn: asym and exp stab}) of the ODE~\eqref{eqn: ODE}. Then the following holds.
		\begin{enumerate}
		\item  For any sequence $\{x_k\}_{k \in \N} \subset ROA_f$ such that $\lim_{k \to \infty} x_k \in \partial ROA_f$ we have that
		\begin{align} \label{eqn: MLF tends to infty}
		\lim_{k \to \infty} V_\beta(x_k) = \infty.
		\end{align}	
		\item We have that
		\begin{equation} \label{eqn: MLF characterizes ROA}
	x \in ROA_f \text{ if and only if } V_{\beta}(x)<\infty.
		\end{equation}
	\end{enumerate}
\end{lem}
\begin{proof}
We first show Statement~1) in Lem.~\ref{lem: MLF} by showing Eqn.~\eqref{eqn: MLF tends to infty} holds. Suppose $\{x_k\}_{k \in \N} \subset ROA_f$ is such that $x^*:=\lim_{k \to \infty} x_k \in \partial ROA_f$. Let $0<\eta_1<\eta$ and consider $T_k:=F_{\eta_1}(x_k)$ (where $F_\eta(x)$ is given in Eqn.~\eqref{fun: set entry times}). Since $x_k \in ROA_f$ it follows $T_k< \infty$ for all $k \in \N$. Moreover, it is clear that $||\phi_f(x_k,t)||_2\ge \eta_1$ for all $t \in [0,T_k)$.

%Since $x_k \in ROA_f$ for each $k \in \N$ it follows that $T_k< \infty$ for each $k \in \N$. Now,
%\begin{align*}
%V(x_k) &= \int_{0}^\infty ||\phi_f(x,t)||_2^2 dt\\
%& = \int_{0}^{T_k} ||\phi_f(x,t)||_2^2 dt + \int_{T_k}^\infty ||\phi_f(x,t)||_2^2 dt\\
%& \le R T_k + \int_{T_k}^\infty \mu \exp(- \delta t) ||x||_2^2\\
%& \le R T_k + \frac{\mu}{\delta}(1 - \exp(-\delta T_k))R < \infty.
%\end{align*}
Now,
\begin{align} \label{pfe 1}
V_\beta(x_k)& = \int_{0}^{T_k} ||\phi_f(x_k,t)||_2^{2 \beta} dt + \int_{T_k}^\infty ||\phi_f(x_k,t)||_2^{2 \beta} dt \\ \nonumber
& \ge \int_{0}^{T_k} ||\phi_f(x_k,t)||_2^{2 \beta} dt \ge \eta_1^{2 \beta} T_k.
\end{align}

We will now show $T_k \to \infty$ as $k \to \infty$ and thus Eqn.~\eqref{pfe 1} shows Eqn.~\eqref{eqn: MLF tends to infty}. For contradiction suppose $\lim_{k \to \infty} T_k \ne 0$, then there exists a bounded subsequence $\{T_{k_n}\}_{n \in \N} \subset \{T_k\}_{k \in \N}$. Now by Theorem~\ref{thm: Bolzano} there exists a subsequence of the subsequence $\{T_{k_n}\}_{n \in \N}$, we denote by $\{T_i\}_{i \in \N}$, that converges to a finite limit $T^*:= \lim_{i \to \infty} T_i< \infty$. Let us denote the corresponding subsequence of $\{x_k\}_{k \in \N}$ by $\{x_i\}_{i \in \N}$. Since $\lim_{k \to \infty} x_k \to x^*$ and every subsequence of a convergent sequence must converge to the same limit we have $\lim_{i \to \infty} x_i = x^*$. %Moreover, by definition we have that $\{T_i\}_{i \in \N}$ is a subsequence to $\{T_k\}_{k \in \N}$.

Since $\phi_f \in C(\R^n \times [0, \infty), \R^n)$ (by Lemma~\ref{lem: diff soln map}) we have that
\begin{align*}
||\phi_f(x^*, T^*)||_2 = \lim_{i \to \infty} ||\phi_f(x_i,T_i)||_2 \le \eta_1 < \eta,
\end{align*}
and since $B_\eta(0)$ is an exponentially stable set we have that
\begin{align} \label{pfe 2}
&||\phi_f(x^*, T^*+t)||_2^2= ||\phi_f(\phi_f(x^*,T^*), t)||_2^2\\ \nonumber
& \qquad \le \mu^2 e^{-2 \delta t} || \phi_f(x^*,T^*)||_2^2  \le \mu^2 \eta^2 e^{-2\delta t}.
\end{align}
Therefore, Eqn.~\eqref{pfe 2} implies that
\begin{align*}
\lim_{t \to \infty} || \phi_f(x^*,t)||_2 & = \lim_{t \to \infty} ||\phi_f(x^*, T^*+t)||_2 = \lim_{t \to \infty} \mu \eta e^{-\delta t}= 0,
\end{align*}
thus showing $x^* \in ROA_f$. Now $ROA_f$ is an open set (by Lemma~\ref{lem: ROA open}). Therefore if $x^* \in ROA_f$ then $x^* \notin \partial ROA_f$, providing a contradiction that $x^* \in \partial ROA_f$. Hence, Eqn.~\eqref{eqn: MLF tends to infty} holds.

We now Statement~2) in Lem.~\ref{lem: MLF} by showing Eqn.~\eqref{eqn: MLF characterizes ROA} holds. First suppose $x \in ROA_f$. We will now show $V_\beta(x)<\infty$. Since $x \in ROA_f$ we have that $\lim_{t \to \infty} ||\phi_f(x,t)||_2 =0$ and thus it follows there exists $T<\infty$ such that $||\phi_f(x,t)||_2< \eta$ for all $t \ge T$ implying  $F_\eta(x)\le T <\infty$. Moreover, by properties of the set entry time we have that $||\phi_f(x,F_\eta(x))||_2 \le \eta$ and since $B_\eta(0)$ is an exponentially stable set we have that,
\begin{align}  \nonumber
||\phi_f(x,t)||_2= ||\phi_f(\phi_f(x,F_\eta(x)), & t-F_\eta(x))||_2 \le \mu  \eta e^{-\delta(t - F_\eta(x) )}\\ \label{pfeqn 1}
&\text{ for all } t>F_\eta(x).
\end{align}
 Therefore, using the fact that $ROA_f \subset B_R(0)$ together with Eqn.~\eqref{pfeqn 1} we get that,
\begin{align*}
& V_\beta(x) = \int_0^{F_\eta(x)} ||\phi_f(x,t)||_2^{2 \beta} dt + \int_{F_\eta(x)}^{\infty} ||\phi_f(x,t)||_2^{2 \beta} dt\\
& \le F_\eta(x) R^{2 \beta} + \mu^{2 \beta} \eta^{2 \beta} \int_{F_\eta(x)}^{\infty} e^{-2 \delta \beta(t - F_\eta(x) )} dt\\
&= F_\eta(x) R^{2 \beta} + \frac{\mu^{2 \beta} \eta^{2 \beta}}{2\delta \beta} < \infty.
\end{align*}

Now, on the other hand let us now suppose $x \in \R^n$ is such that $V_\beta(x)< \infty$. We will show $x \in ROA_f$. For contradiction suppose $x \notin ROA_f$. Then $\lim_{t \to \infty}||\phi_f(x,t)||_2 \ne 0$. Therefore, there exists $\eps>0$ such that $||\phi_f(x,t)||_2> \eps$ for all $t \ge 0$. Thus
\begin{align*}
V_\beta(x)= \int_0^\infty ||\phi_f(x,t)||_2^{2 \beta} dt \ge \int_0^\infty  \eps^{2 \beta} dt = \infty,
\end{align*}
providing a contradiction that $V_\beta(x)< \infty$. Hence, Eqn.~\eqref{eqn: MLF characterizes ROA} holds.
\end{proof}
As we will show next, the function, $W_{\lambda,\beta}$, given in Eqn.~\eqref{fun: ZLF}, can also characterize $ROA_f$ as its $1$-sublevel set.
\begin{cor} \label{cor: ZLF characterizes ROA}
	Consider $f \in LocLip(\R^n, \R)$, $\beta \in \N$, $\lambda>0$ and $W_{\lambda,\beta}$ given in Eqn.~\eqref{fun: ZLF}. Suppose there exists $R,\eta>0$ such that $ROA_f \subset B_R(0)$ and $B_\eta(0)$ is an exponentially stable set (Defn.~\ref{defn: asym and exp stab}) to the ODE~\eqref{eqn: ODE}. Then the following holds.
\begin{enumerate}
	\item  For any sequence $\{x_k\}_{k \in \N} \subset ROA_f$ such that $\lim_{k \to \infty} x_k \in \partial ROA_f$ we have that
	\begin{align} \label{eqn: ZLF tends to 1}
	\lim_{k \to \infty} W_{\lambda,\beta}(x_k) = 1.
	\end{align}	
	\item We have that
	\begin{equation} \label{eqn: ZLF characterizes ROA}
	 ROA_f = \{x\in \R^n: W_{\lambda,\beta}(x)<1 \}.
	\end{equation}
\end{enumerate}
\end{cor}
\begin{proof}
	We first show Statement~1) in Cor.~\ref{cor: ZLF characterizes ROA} by showing Eqn.~\eqref{eqn: ZLF tends to 1} holds. For $x \in ROA_f$ we have that $W_{\lambda,\beta}(x) = 1 - e^{-\lambda V_\beta(x)}$. Moreover, $e^x$ is a continuous function of $x \in \R$. Therefore, by Lemma~\ref{lem: MLF}, for $\{x_k\}_{k \in \N} \subset ROA_f$ we have that \begin{align*}
	\lim_{k \to \infty} W_{\lambda,\beta}(x_k)= 1 - \exp \left({-\lambda \lim_{k \to \infty} V_\beta(x_k)} \right) =1.
	\end{align*}
	
		We next show Statement~2) in Cor.~\ref{cor: ZLF characterizes ROA} by showing Eqn.~\eqref{eqn: ZLF characterizes ROA} holds. If $x \in ROA_f$ then by Lemma~\ref{lem: MLF} we have that $V_\beta(x) < \infty$ and thus $e^{- \lambda V_\beta(x)} >0$ implying $W_{\lambda, \beta}(x)= 1 - e^{-\lambda V_\beta(x)}<1$. Therefore, $ROA_f \subseteq \{x \in \R^n: W_{\lambda, \beta}(x)< 1  \}$. On the other hand if $y \in \{x \in \R^n: W_{\lambda, \beta}(x)< 1  \}$ then $a:= 1 - W_{\lambda,\beta}(y)>0$. Thus,  $V_\beta(y) = - \frac{1}{\lambda} \ln(a)< \infty$. Lemma~\ref{lem: MLF} shows if $V_\beta(y) < \infty$ then $y \in ROA_f$. Hence, $ \{x \in \R^n: W_{\lambda, \beta}(x)< 1  \} \subseteq ROA_f $.
\end{proof}

\subsection{A Globally Lipschitz Continuous Lyapunov Function}
The function $V_\beta$ is only defined over the set $ROA_f$ and is unbounded. Such properties make approximating $V_\beta$ by polynomials challenging. On the other hand $W_{\lambda,\beta}$ is defined over the whole of $\R^n$ and is bounded by $1$. What is more, we next show in Prop.~\ref{prop: ZLF is lip} that $W_{\lambda,\beta}$ is globally Lipschitz continuous. One may intuit this continuity property by considering the similarity in structure between $W_{\lambda, \beta}$ and the standard mollifier given in Eqn.~\eqref{fun: mollifier}; a function known to be infinitely differentiable.
\begin{prop}\label{prop: ZLF is lip}
Consider $f \in C^2(\R^n, \R)$ and $W_{\lambda, \beta}$ as in Eqn.~\eqref{fun: ZLF} where $\lambda>0$ and $\beta \in \N$. Suppose there exists $\theta,\eta,R>0$ such that $||D^\alpha f(x)||_2<\theta$ for all $x \in B_R(0)$ and $||\alpha||_1 \le 2$, $B_\eta(0)$ is an exponentially stable set (Defn.~\ref{defn: asym and exp stab}) to the ODE~\eqref{eqn: ODE}, and $ROA_f \subset B_R(0)$. Then if $\lambda > \theta \eta^{-2 \beta}$ and $ \beta>\frac{\theta}{2\delta} +\frac{1}{2}$ we have that $W_{\lambda, \beta} \in Lip(\R^n, \R)$. Moreover, the Lipschitz constant of $W_{\lambda, \beta}$ is less than or equal to $K>0$, where
\begin{align} \label{eqn: lip constant of ZLF}
K:=2 \lambda \max \left\{\frac{2 \beta R^{2 \beta -1}}{\theta}, \frac{2 \beta (\mu \eta)^{2 \beta -1}}{\delta(2\beta -1) - \theta} \right\}.
\end{align}
\end{prop}
\begin{proof}%[Proof of Lemma~\ref{lem: ZLF is lip}]
%	We split the remainder of the proof into four disjoint cases.
To prove $W_{\lambda,\beta} \in Lip(\R^n, \R)$ we will now show
\begin{align} \label{lip}
|W_{\lambda,\beta}(x) - W_{\lambda,\beta}(y)|<K ||x-y||_2 \text{ for all } x,y \in \R^n,
\end{align}
where $K>0$ is given in Eqn.~\eqref{eqn: lip constant of ZLF}.

	{\textbf{Case 1: $x,y \in ROA_f$.}} Since $B_\eta(0)$ is an exponentially stable set of the ODE~\eqref{eqn: ODE} and by applying a similar argument in the derivation of Eqn.~\eqref{pfeqn 1}, it follows that there exists $\delta,\mu>0$ such that
	\begin{align} \label{pfeq: 1}
	& ||\phi_f(x,t)||_2 \le \mu \eta e^{-\delta (t- F_\eta(x))}  \text{ for all } t>F_\eta(x),\\ \nonumber
	& ||\phi_f(y,t)||_2 \le \mu \eta e^{-\delta (t- F_\eta(x))} \text{ for all } t>F_\eta(y).
	\end{align}
	%Furthermore, let us denote $r:=||x-y||_2$,  $\tilde{T}_y:=\inf\{t \ge 0: \phi_f(y,t) \in B_r(0)\}$ and $\tilde{T}_x:=\inf\{t \ge 0: \phi_f(y,t) \in B_r(0)\}$.
	Without loss of generality we will assume $F_\eta(x) \ge F_\eta(y)$ (otherwise we can relabel $x$ and $y$).
	
	Now,
	\begin{align} \label{pfeqn:lip}
	& |W_{\lambda, \beta }(x) - W_{\lambda, \beta }(y)|= \bigg|\exp \left(- \lambda \int_{0}^\infty ||\phi_f(x,t)||_2^{2 \beta} dt \right)  \\ \nonumber
	& \hspace{3cm} -\exp \left(- \lambda \int_{0}^\infty ||\phi_f(y,t)||_2^{2 \beta} dt \right) \bigg| \\ \nonumber
	& = \bigg| \exp \left(- \lambda \int_{0}^\infty ||\phi_f(x,t)||_2^{2 \beta} dt \right) \bigg|  \\ \nonumber
	& \qquad \times \bigg| 1- \exp \left(- \lambda \int_{0}^\infty (||\phi_f(y,t)||_2^{2 \beta} - ||\phi_f(x,t)||_2^{2 \beta})  dt \right) \bigg |\\ \nonumber
	& \le \bigg| \exp \left(- \lambda \int_{0}^\infty ||\phi_f(x,t)||_2^{2 \beta} dt \right) \bigg|\\ \nonumber
	& \qquad \times \bigg| \lambda \int_{0}^\infty (||\phi_f(y,t)||_2^{2 \beta} - ||\phi_f(x,t)||_2^{2 \beta})  dt \bigg| \\ \nonumber
	& = \lambda  \exp \left(- \lambda V_\beta(x) \right)  \left|V_\beta(x) - V_\beta(y) \right|,
	\end{align}
	where the inequality in Eqn.~\eqref{pfeqn:lip} follows by the exponential inequality given in Eqn.~\eqref{ineq:exp 3} in Lemma~\ref{lem: exp ineq} and the function $V_\beta$ is as in Eqn.~\eqref{fun: MLF}.
	
%		We now split the remainder of the proof into two further sub-cases. Before doing so, for convenience, let us first denote the constant $\alpha := \frac{\delta}{\delta \beta + 2\theta}$.
%	
%	\textbf{Case 1a: $||x-y||_2^\alpha < \eta$.}

		We first derive a bound for $|V_{\beta}(x)-V_{\beta}(y)|$.
	\begin{align} \label{pfeqn 4}
	&|V_{\beta}(x) - V_{\beta}(y)| = \left| \int_{0}^{\infty} ||\phi_f(x,t)||_2^{2 \beta} - ||\phi_f(y,t)||_2^{2 \beta}  dt \right|   \\ \nonumber
	 & \le  \int_{0}^{\infty} \bigg| ||\phi_f(x,t)||_2 - ||\phi_f(y,t)||_2 \bigg| \\ \nonumber
	& \qquad \qquad \times \left( \sum_{i=0}^{2 \beta-1} ||\phi_f(x,t)||^i_2 ||\phi_f(y,t)||^{2 \beta - 1 -i}_2  \right) dt  \\ \nonumber
	& \le \int_{0}^{F_\eta(x)} \bigg| ||\phi_f(x,t) - \phi_f(y,t)||_2 \bigg| \\ \nonumber
	& \qquad \qquad \times \left( \sum_{i=0}^{2 \beta-1} ||\phi_f(x,t)||^i_2 ||\phi_f(y,t)||^{2 \beta - 1 -i}_2  \right) dt\\ \nonumber
	& \qquad + \int_{F_\eta(x)}^{\infty} \bigg| ||\phi_f(x,t) - \phi_f(y,t)||_2 \bigg| \\ \nonumber
	& \qquad \qquad \qquad \times \left( \sum_{i=0}^{2 \beta-1} ||\phi_f(x,t)||^i_2 ||\phi_f(y,t)||^{2 \beta - 1 -i}_2  \right) dt.
	\end{align}
	We now derive a bound for the two terms that appear in the right hand side of Eqn.~\eqref{pfeqn 4}. Using the fact $||\phi_f(x,t)||_2 <R$ and $||\phi_f(y,t)||_2 <R$ since $ROA_f \subset B_R(0)$, and using Lemma~\ref{lem: exp divergence of solution maps}, we get,
	\begin{align} \label{pfeqn 5}
&\int_{0}^{F_\eta(x)}  ||\phi_f(x,t) - \phi_f(y,t)||_2  \\ \nonumber
& \qquad \qquad \times \left( \sum_{i=0}^{2 \beta-1} ||\phi_f(x,t)||^i_2 ||\phi_f(y,t)||^{2 \beta - 1 -i}_2  \right) dt \\ \nonumber
& \le 2 \beta R^{2 \beta -1} \int_{0}^{F_\eta(x)}  ||\phi_f(x,t) - \phi_f(y,t)||_2  dt \\ \nonumber
& \le 2 \beta R^{2 \beta -1} ||x-y||_2 \int_{0}^{F_\eta(x)} e^{\theta t} dt\\ \nonumber
& = \frac{2 \beta R^{2 \beta -1}}{\theta} \left( e^{\theta F_\eta(x)} -1 \right) ||x-y||_2.
	\end{align}

Moreover, since $ \beta>\frac{\theta}{2\delta} +\frac{1}{2}$ it also follows using Eqn.~\eqref{pfeq: 1}, and Lemma~\ref{lem: exp divergence of solution maps}, that
		\begin{align} \label{pfeqn 6}
	&\int_{F_\eta(x)}^\infty  ||\phi_f(x,t) - \phi_f(y,t)||_2  \\ \nonumber
	& \qquad \qquad \times \left( \sum_{i=0}^{2 \beta-1} ||\phi_f(x,t)||^i_2 ||\phi_f(y,t)||^{2 \beta - 1 -i}_2  \right) dt \\ \nonumber
	& \le 2 \beta (\mu \eta)^{2 \beta -1} e^{\delta (2 \beta -1) F_\eta(x)} ||x-y||_2\\ \nonumber
	& \hspace{3cm} \times \int_{F_\eta(x)}^\infty e^{\theta t- \delta(2 \beta -1) t} dt\\ \nonumber
	& = \frac{2 \beta (\mu \eta)^{2 \beta -1}}{\delta(2\beta -1) - \theta} e^{\theta F_\eta(x)} ||x-y||_2.
	\end{align}
	Now, combining Eqns~\eqref{pfeqn 4}, \eqref{pfeqn 5} and \eqref{pfeqn 6} we get,
		\begin{align} \label{pfeqn 10}
	&|V_{\beta}(x) - V_{\beta}(y)| \\ \nonumber
	&\le \max \left\{\frac{2 \beta R^{2 \beta -1}}{\theta}, \frac{2 \beta (\mu \eta)^{2 \beta -1}}{\delta(2\beta -1) - \theta} \right\} e^{\theta F_\eta(x)}   ||x-y||_2.
	\end{align}
	We next derive a bound for the $\exp \left(- \lambda V_\beta(x) \right) $ term in Eqn.~\eqref{pfeqn:lip}.
	\begin{align} \label{pfeqn 11}
& \exp \left(- \lambda V_\beta(x) \right)  = \exp \left(- \lambda \int_0^\infty ||\phi_f(x,t)||_2^{2 \beta} dt \right)\\ \nonumber
& \qquad \le \exp \left(- \lambda \int_0^{F_\eta(x)} ||\phi_f(x,t)||_2^{2 \beta} dt \right) \le e^{- \lambda F_\eta(x) \eta^{2 \beta} }.
	\end{align}
	Finally combining Eqns~\eqref{pfeqn:lip}, \eqref{pfeqn 10}, and \eqref{pfeqn 11}, and using the fact $\lambda> \theta \eta^{-2 \beta}$, we get
	\begin{align*}
& |W_{\lambda, \beta }(x) - W_{\lambda, \beta }(y)| \\
& \le \lambda \max \left\{\frac{2 \beta R^{2 \beta -1}}{\theta}, \frac{2 \beta (\mu \eta)^{2 \beta -1}}{\delta(2\beta -1) - \theta} \right\} \\
& \qquad \qquad \qquad \qquad \qquad \qquad \times e^{ - (\lambda  \eta^{2 \beta} -\theta) F_\eta(x)}   ||x-y||_2 \\
& \le \frac{K}{2} ||x-y||_2,
	\end{align*}
	showing Eqn.~\eqref{lip} holds when $x,y \in ROA_f$.

{\textbf{Case 2: $x \in ROA_f$ and $y \notin ROA_f$.}} Let us consider the set $\{z_\beta\}_{\beta \in [0,1]}\subset \R^n$ where for $\beta \in [0,1]$ we have that $z_\beta:= (1-\beta)x + \beta y$. Now, since $x \in ROA_f$ and $ROA_f$ is open (by Lemma~\ref{lem: ROA open}) it follows there exists $\eps>0$ such that $B_\eps(x) \subset ROA_f$. Therefore, since $||z_\beta - x||= |\beta| ||x - y||_2$, it follows $z_\beta   \in ROA_f$ for all $\beta \in [0,\eps/||x-y||_2)$. Thus $\sigma:=\sup\{\beta: z_\beta \in ROA_f \}\ge \eps/||x-y||_2>0$. Moreover, $\sigma\le 1$ as $z_1= y \notin ROA_f$.

Consider $a_n:= \sigma (1- 1/n)$ and denote the sequence of points $w_n:=z_{a_n}$. It follows $\{w_n\}_{n \in \N} \subset ROA_f$ and $w^*:= \lim_{n \to \infty} w_n \in \partial ROA_f$. By Lemma~\ref{lem: MLF} we have that $\lim_{n \to \infty}V_\beta(w_n)=\infty$. Therefore there exists $N \in \N$ such that
\begin{align} \label{11}
\exp(-\lambda V_\beta(w_n)) < \frac{K}{2} ||x-y||_2 \text{ for all } n >N.
\end{align}
Moreover, since $y \notin ROA_f$ we have that $W_{\lambda, \beta}(y)=1$. Thus by Eqn.~\eqref{11} we have that
\begin{align} \label{111}
& | W_{\lambda, \beta}(w_n) - W_{\lambda, \beta}(y)| = |1- \exp(-\lambda V_\beta(w_n)) -1|\\ \nonumber
& \qquad = \exp(-\lambda V_\beta(w_n)) \le \frac{K}{2} ||x-y||_2 \text{ for all } n >N.
\end{align}

Furthermore, for any $n>N$ we have that $w_n \in ROA_f$ and $x \in ROA_f$ and thus Case~1 shows that \begin{align} \label{12}
|W_{\lambda, \beta}(x) - W_{\lambda, \beta}(w_n)|<\frac{K}{2} ||x-w_n||_2.
\end{align}
Thus, by Eqns~\eqref{111} and \eqref{12} and selecting any $n >N$, it now follows that
\begin{align*}
|W_{\lambda, \beta}(x) & - W_{\lambda, \beta}(y)| \\
&=|W_{\lambda, \beta}(x) - W_{\lambda, \beta}(w_n)| + | W_{\lambda, \beta}(w_n) - W_{\lambda, \beta}(y)|\\
& \le \frac{K}{2} ||x- w_n||_2 + \exp(-\lambda V(w_n))\\
& \le \frac{K}{2} \sigma \left(1 - \frac{1}{n} \right)  ||x-y||_2 + \frac{K}{2} ||x-y||_2\\
& \le K ||x-y||_2,
\end{align*}
where the third inequality follows since $\sigma(1 - \frac{1}{n})<1$ for all $n \in \N$. Therefore, Eqn.~\eqref{lip} holds when $x \in ROA_f$ and $y \notin ROA_f$.

{\textbf{Case 3: $y \in ROA_f$ and $x \notin ROA_f$.}} It follows by a similar argument to Case 2 that
\begin{align*}
|W_{\lambda,\beta}(x) - W_{\lambda,\beta}(y)| \le K||x-y||_2,
\end{align*}
and thus Eqn.~\eqref{lip} holds when $y \in ROA_f$ and $x \notin ROA_f$.

{\textbf{Case 4: $x,y \notin ROA_f$.}} We have that $W_{\lambda,\beta}(x)=W_{\lambda,\beta}(y)=1$ for all $x,y \notin ROA_f$ and thus,
\begin{align*}
|W_{\lambda,\beta}(x) - W_{\lambda,\beta}(y)|= 0 \le K||x-y||_2,
\end{align*} and thus Eqn.~\eqref{lip} holds when $x,y \notin ROA_f$. \end{proof}

\subsection{The Converse Lyapunov Function Satisfies a PDE}
Proposition~\ref{prop: ZLF is lip} shows $W_{\lambda,\beta}$ is a Lipschitz continuous function when $\lambda>0$ and $\beta \in \N$ are sufficiently large. Rademacher's Theorem (Theorem~\ref{thm: Rademacher theorem} found in Appendix~\ref{sec: appendix}) shows that Lipschitz continuous functions are differentiable almost everywhere. Therefore, $W_{\lambda,\beta}$ must satisfy some Partial Differential Equation (PDE) almost everywhere. We next derive this PDE by showing $W_{\lambda,\beta}$ satisfies Eqn.~\eqref{PDE: ZLF}.
\begin{prop} \label{prop: ZLF PDE}
	Consider $f \in C^2(\R^n, \R)$ and $W$ as in Eqn.~\eqref{fun: ZLF}. Suppose there exists $\theta,\eta,R>0$ such that $||D^\alpha f(x)||_2<\theta$ for all $x \in B_R(x)$ and $||\alpha||_1 \le 2$, $B_\eta(0)$ is an exponentially stable set (Defn.~\ref{defn: asym and exp stab}) of the ODE~\eqref{eqn: ODE}, and $ROA_f \subset B_R(0)$. If $\lambda > \theta \eta^{-2 \beta}$ and $ \beta>\frac{\theta}{2\delta} +\frac{1}{2}$ then
	\begin{align} \label{PDE: ZLF}
	\nabla W_{\lambda, \beta}(x)^T f(x) = -\lambda ||x||_2^{2 \beta} & (1 - W_{\lambda, \beta}(x))\\ \nonumber
	&  \text{ for almost every } x \in \R^n.
	\end{align}
\end{prop}
\begin{proof}
	By Prop.~\ref{prop: ZLF is lip} we have that $W_{\lambda, \beta} \in Lip(\R^n, \R)$. Therefore by Rademacher's Theorem (stated in Theorem~\ref{thm: Rademacher theorem} and found in Appendix~\ref{sec: appendix}) $W_{\lambda, \beta}$ is differentiable almost everywhere. Moreover, $\phi_f$ is differentiable by Lemma~\ref{lem: diff soln map}. Since the composition of differentiable functions is itself differentiable it follows by the chain rule that,
	\begin{align} \label{pe 1}
	\frac{d}{dt} W_{\lambda, \beta}(\phi_f(x,t)) & \bigg|_{t=0} = \nabla W_{\lambda, \beta}(x)^T \frac{\partial}{\partial t} \phi_f(x,t) \bigg|_{t=0} \\
	\nonumber
	&= \nabla W_{\lambda, \beta}(x)^Tf(x)
	\text{ for almost every }  x \in \R^n.
	\end{align}
	On the other hand, if $x \in ROA_f$ it follows $\phi_f(x,t) \in ROA_f$ for all $t \ge 0$ and thus,
	\begin{align} \label{pfe: W comp phi}
	W_{\lambda, \beta}(\phi_f(x,t))=1 - \exp \left(-\lambda \int_{t}^\infty ||\phi_f(x,s)||_2^{2 \beta} ds \right).
	\end{align}
%	Now, for $x \in ROA_f$ it is clear that $T_x:= \sup\{T \ge 0: ||\phi_f(x,t)||_2 > \eta\}< \infty$. Therefore, since $B_\eta(0)$ is an exponentially stable set it follows that for all $x \in ROA_f$ we have that
%	\begin{align*}
%	& \int_{0}^\infty ||\phi_f(x,t)||_2^2 dt = \int_{0}^{T_k} ||\phi_f(x,t)||_2^2 dt + \int_{T_x}^\infty ||\phi_f(x,t)||_2^2 dt\\
%	& \le R T_x + \int_{T_x}^\infty \mu \exp(- \delta t) ||x||_2^2\\
%	& \le R T_x + \frac{\mu}{\delta}(1 - \exp(-\delta T_x))R < \infty,
%	\end{align*}
%	which shows the integral in Eqn.~\eqref{pfe: W comp phi} is finite. Therefore
	By the fundamental theorem of calculus and the fact $\phi_f(x,0)=x$ for all $x \in \R^n$ we have that,
	\begin{align} \nonumber
	&\frac{d}{dt} W_{\lambda, \beta}(\phi_f(x,t)) \bigg|_{t=0} \hspace{-0.5cm} \\ \nonumber
	&= -\lambda ||\phi_f(x,t)||_2^{2 \beta} \exp \left(-\lambda \int_{t}^\infty ||\phi_f(x,s)||_2^{2 \beta} ds \right) \bigg|_{t=0}\\ \label{pe 2}
	&  =  -\lambda ||x||_2^{2 \beta}(1- W_{\lambda, \beta}(x)) \text{ for } x \in ROA_f.
	\end{align}
	If $x \notin ROA_f$ then clearly $\phi_f(x,t) \notin ROA_f$ for all $t \ge 0$. Thus $W(\phi_f(x,t))= 1$ for all $x \notin ROA_f$ and $t \ge 0$. Therefore,
	\begin{align} \label{pe 3}
	&\frac{d}{dt} W_{\lambda, \beta}(\phi_f(x,t)) \bigg|_{t=0} \hspace{-0.25cm} = \frac{d}{dt} 1 \bigg|_{t=0} \hspace{-0.25cm} = 0 = -\lambda ||x||_2^{2 \beta}(1-1)  \\ \nonumber
	&  \qquad = -\lambda ||x||_2^{2 \beta}(1-W_{\lambda, \beta}(x)) \text{ for } x \notin ROA_f.
	\end{align}
	Hence, Eqns~\eqref{pe 1}, \eqref{pe 2} and \eqref{pe 3} prove that the PDE given in Eqn.~\eqref{PDE: ZLF} holds. \end{proof}

\section{A Convex Optimization Problem for Approximating the Converse Lyapunov Function} \label{sec: bounding LF from above}
We have reduced the problem of approximating the region of attraction to solving the optimization problem given in Eqn.~\eqref{opt: algebriac ROA}, where $W=W_{\lambda,\beta}$ is given in Eqn.~\eqref{fun: ZLF}. Unfortunately, no analytical formula for $W_{\lambda,\beta}$ is known. Therefore, the optimization problem given in Eqn.~\eqref{opt: algebriac ROA} cannot be solved in its current form.

Fortunately, the unknown function $W_{\lambda,\beta}$ can be removed from the objective function of Opt.~\eqref{opt: algebriac ROA}. To see this note that if $J(x) \ge W_{\lambda,\beta}(x)$ for all $x \in \Lambda \subseteq \Omega$, then minimizing $\int_{\Lambda} |J(x)-W_{\lambda,\beta}(x)| dx$ is equivalent to minimizing $\int_{\Lambda} J(x) dx$. Thus, Opt.~\eqref{opt: algebriac ROA} is equivalent to the following optimization problem,
\vspace{-0.1cm}
\begin{align} \label{opt: intermediate}
& P_d\in \arg \min_{J \in \mcl P_d(\R^n,\R)} \int_{\Lambda} J(x) dx\\ \nonumber
& \text{ such that } J(x) \ge W_{\lambda,\beta}(x) \text{ for all } x \in \Omega.
\end{align}

Unfortunately, the constraint of Opt.~\eqref{opt: intermediate} still involves the unknown function $W_{\lambda,\beta}$. In the absence of an analytical formula for $W_{\lambda,\beta}$ we propose in Prop.~\ref{prop: diss ineq gives super LF} conditions, in the form of the linear partial differential inequalities given in Eqns~\eqref{eqn: diss }, \eqref{eqn: diss 2 } and \eqref{eqn: diss 3}, that when satisfied by some function $J\in C^1(\R^n, \R)$ implies that $W_{\lambda,\beta}(x) \le J(x)$. Thus, any $J$ satisfying Eqns~\eqref{eqn: diss }, \eqref{eqn: diss 2 } and \eqref{eqn: diss 3} is feasible to Opt.~\eqref{opt: intermediate}. %Therefore, Opt.~\eqref{opt: intermediate} can be tightened to a convex optimization problem (given later in Eqn.~\eqref{opt: tightened algebriac ROA}) by replacing the constraint $W_{\lambda,\beta}(x) \le J(x)$ with Eqns~\eqref{eqn: diss }, \eqref{eqn: diss 2 } and \eqref{eqn: diss 3}.

\subsection{Bounding The Converse Lyapunov Function From Above}
%In Prop.~\ref{prop: ZLF PDE} it was shown that the converse Lyapunov function $W_{\lambda,\beta}$ satisfies the PDE~\eqref{PDE: ZLF}. Rather than solving PDE~\eqref{PDE: ZLF}, to compute $W_{\lambda,\beta}$ directly, we relax the PDE~\eqref{PDE: ZLF} to LPDIs~\eqref{eqn: diss }, \eqref{eqn: diss 2 } and \eqref{eqn: diss 3}. We then show that any function that satisfies the LPDIs uniformly bounds $W_{\lambda,\beta}(x) $ from above.
\begin{prop} \label{prop: diss ineq gives super LF}
	Consider $f \in C^1(\R^n, \R)$, $\beta \in \N$ and $\lambda>0$. Suppose there exists $J \in C^1(\Omega, \R)$ that satisfies
	\begin{align} \label{eqn: diss }
	& \nabla J(x)^T f(x) \le  -\lambda||x||_2^{2 \beta} (1-J(x)) \text{ for all } x \in \Omega, \\
	& J(x) \ge 1 \text{ for all } x \in \partial \Omega, \label{eqn: diss 2 }\\
	& J(0) \ge 0, \label{eqn: diss 3}
	\end{align}
	where $\Omega \subset \R^n$ is a compact set. Then $W_{\lambda, \beta}(x) \le J(x)$ for all $x \in \Omega$, where $W_{\lambda, \beta}$ is as in Eqn.~\eqref{fun: ZLF}.
\end{prop}
\begin{proof}
	Consider $x \in \Omega$. Let us consider the time the solution map exits the set $\Omega\subset \R^n$, denoted by ${T_x}:= \sup\{t \ge 0: \phi_f(x,t) \in \Omega \}$. Furthermore, let us denote $u(t):= J(\phi_f(x,t))-1$ and $\alpha(t):=\lambda ||\phi_f(x,t)||_2^{2 \beta}$. It follows from Eqn.~\eqref{eqn: diss } that
	\begin{align*}
	\frac{d}{dt}u(t) \le \alpha(t) u(t) \text{ for all } t \in [0,T_x].
	\end{align*}
	Therefore by Lemma~\ref{lem: gronwall} it follows that
	\begin{align*}
	u(t) \le u(0) \exp\left(\int_0^t \alpha(s) ds\right) \text{ for all } t \in [0,T_x],
	\end{align*}
	and thus selecting $t =T_x$ we have that
	\begin{align} \label{1}
	J(\phi_f(x,T_x))-1 \le (J(x)-1) \exp\left(\lambda \int_0^{T_x} ||\phi_f(x,s)||_2^{2 \beta} ds\right).
	\end{align}
	By rearranging Eqn.~\eqref{1} we get that,
	\begin{align} \label{2}
	J(x) & \ge 1 - (1-J(\phi_f(x,T_x)))\exp\left(- \lambda \int_0^{T_x} ||\phi_f(x,s)||_2^{2 \beta} ds\right) .
	\end{align}
	\textbf{Case 1: $T_x< \infty$.} In this case the solution map exits the set $\Omega$ in some finite time. Since $\phi_f \in C(\R^n \times [0,\infty), \R^n)$ (by Lemma~\ref{lem: diff soln map}) it is clear that $\phi_f(x,T_x) \in \partial \Omega$. Therefore by Eqn.~\eqref{eqn: diss 2 } we have that $J(\phi_f(x,T_x)) \ge 1$. Hence, $(1-J(\phi_f(x,T_x)))\exp\left(- \lambda \int_0^{T_x} ||\phi_f(x,s)||_2^{2 \beta} ds\right) \le 0$. Thus, by Eqn.~\eqref{2} we have that,
	\begin{align*}
	J(x) & \ge 1 - (1-J(\phi_f(x,T_x)))\exp\left(- \lambda \int_0^{T_x} ||\phi_f(x,s)||_2^{2 \beta} ds\right) \\
	& \ge 1 \ge W_{\lambda,\beta}(x),
	\end{align*}
	{since $W_{\lambda,\beta}(x) \le 1$}.
	
	\textbf{Case 2a: $T_x = \infty$ and $x \in ROA_f$.} In this case we have $\lim_{t \to \infty} ||\phi_f(x,t)||_2=||\phi_f(x,T_x)||_2=0$ since $x \in ROA_f$. Moreover, since $J(\phi_f(x,T_x))=J(0) \ge 0$ (by Eqn.~\eqref{eqn: diss 3}) and $\exp(x) \ge 0$ for all $x \in \R$ it follows from Eqn.~\eqref{2} that
	\begin{align*}
	J(x) & \ge 1 - (1-J(0))\exp\left(-\lambda \int_0^{\infty} ||\phi_f(x,s)||_2^{2 \beta} ds\right)\\
	& \ge 1 - \exp\left(- \lambda \int_0^{\infty} ||\phi_f(x,s)||_2^{2 \beta} ds\right) = W_{\lambda,\beta}(x).
	\end{align*}
%	\textbf{Case 2: $x \in \Omega/ ROA_f$ and $T_x<\infty$}. If $x \in \Omega/ ROA_f$ we have that $W(x)=1$. Moreover, if $T_x< \infty$ then the solution map exits the set $\Omega$ in finite time, that is $\phi_f(x,T_x) \in \partial \Omega$. Therefore it follows from Eqn.~\eqref{eqn: diss 2 } that $J(\phi_f(x,T_x)) \ge 1$. Now since $\exp(x) \ge 0$ and $-(1-J(\phi_f(x,T_x))) \ge 0$ it follows $- (1-J(\phi_f(x,T_x)))\exp\left(-\int_0^{T_x} ||\phi_f(x,s)||_2^2 ds\right) \ge 0$. Hence, by Eqn.~\eqref{2} we have that
%	\begin{align*}
%	J(x) & \ge 1 - (1-J(\phi_f(x,T_x)))\exp\left(-\lambda \int_0^{T_x} ||\phi_f(x,s)||_2^{2 \beta} ds\right)\\
%	& \ge 1 = W_{\lambda, \beta}(x).
%	\end{align*}
	\textbf{Case 2b: $T_x=\infty$ and $x \in \Omega/ ROA_f$}. If $x \in \Omega/ ROA_f$ we have that $W(x)=1$. Moreover, if $T_x= \infty$ then the solution map never exits the set $\Omega$, that is $\phi_f(x,t) \in \Omega$ for all $t \ge 0$. Since $J$ is differentiable and $\Omega$ is compact we have that $J$ is bounded, that is, there exists $M>0$ such that $|J(\phi_f(x,t))|<M$ for all $t \ge 0$. Since $x \notin ROA_f$ we have that $\phi_f(x,t) \notin ROA_f$ for all $t \ge 0$. This there exists $\eps>0$ such that $||\phi_f(x,t)||_2^{2 \beta} \ge \eps^{2 \beta} $ for all $t \ge 0$. Thus, since $|J(\phi_f(x,t))|<M$ for all $t \ge 0$, we have that
	\begin{align*}
	& \left| (1-J(\phi_f(x,T_x)))\exp\left(- \lambda \int_0^{T_x} ||\phi_f(x,s)||_2^{2 \beta} ds\right) \right|\\
	& = \lim_{T \to \infty} \left| \bigg(1-J(\phi_f(x,T)) \bigg) \exp\left(- \lambda \int_0^{T} ||\phi_f(x,s)||_2^{2 \beta} ds\right) \right| \\
		& = \lim_{T \to \infty} \left| 1-J(\phi_f(x,T)) \right| \exp\left(- \lambda \int_0^{T} ||\phi_f(x,s)||_2^{2 \beta} ds\right)  \\
	&\le \lim_{T \to \infty} \bigg\{(M+1)\exp\left(-T \lambda \eps^{2 \beta} \right) \bigg\}=0,
	\end{align*}
	implying \hspace{-0.05cm}$(1- \hspace{-0.05cm}J(\phi_f(x,T_x)))\exp\left( \hspace{-0.05cm}- \lambda \int_0^{T_x} ||\phi_f(x,s)||_2^{2 \beta} \hspace{-0.05cm} ds\right) \hspace{-0.05cm}= \hspace{-0.05cm}0$.
	
	It is now clear by Eqn.~\eqref{2} that
	\begin{align*}
	J(x) & \ge 1 - (1-J(\phi_f(x,T_x)))\exp\left(-\lambda \int_0^{T_x} ||\phi_f(x,s)||_2^{2 \beta} ds\right)\\
	& \ge 1 = W_{\lambda, \beta}(x).
	\end{align*}
\end{proof}

\begin{cor} \label{cor: diss ineq imply positivity}
	Consider $f \in C^1(\R^n, \R)$, $\beta \in \N$ and $\lambda>0$. Suppose there exists $J \in C^1(\Omega, \R)$ that satisfies Eqns~\eqref{eqn: diss }, \eqref{eqn: diss 2 } and \eqref{eqn: diss 3} for some compact set $\Omega$. Then $J(x) \ge 0$ for all $x \in \Omega$.
\end{cor}
\begin{proof}
	By Prop.~\ref{prop: diss ineq gives super LF} we have that $J(x) \ge W_{\lambda,\beta}(x) \ge 0$, where $W_{\lambda,\beta}$ is as in Eqn.~\eqref{fun: ZLF}.
\end{proof}

\subsection{Tightening The Problem of Approximating Our Proposed Converse Lyapunov Function}
Using Prop.~\ref{prop: diss ineq gives super LF} we now tighten the optimization problem given in Eqn.~\eqref{opt: intermediate}. For given $f \in C^1(\R^n, \R^n)$, $\lambda>0$, $\beta \in \N$, $R>0$ and $\Lambda \subseteq \Omega \subset \R^n$ consider the following optimization problem,
\vspace{-0.2cm}
\begin{align} \label{opt: tightened algebriac ROA}
& P_d\in \arg \min_{J \in \mcl P_d(\R^n,\R)} \int_{\Lambda} J(x) dx\\ \nonumber
& \text{ such that } J \text{ satisfies } \eqref{eqn: diss }, \eqref{eqn: diss 2 }, \text{ and } \eqref{eqn: diss 3}.
\end{align}
Clearly the Opt.~\eqref{opt: tightened algebriac ROA} is a tightening of the Opt.~\eqref{opt: algebriac ROA} since if $J$ is feasible to Opt.~\eqref{opt: tightened algebriac ROA} then by Prop.~\ref{prop: diss ineq gives super LF} we have that $J$ is also feasible to Opt.~\eqref{opt: algebriac ROA}. Moreover, Opt.~\eqref{opt: tightened algebriac ROA} is a convex optimization problem since it is linear in its decision variable, $J$, in both the constraints and objective function. In the next section we further tighten Opt.~\eqref{opt: tightened algebriac ROA} to an SOS Optimization Problem~\eqref{opt: SOS for ROA} that can be tractably solved. For implementation purposes we select $\Omega = B_R(0)$, where $R>0$, and $\Lambda \subseteq \Omega$ as some rectangular set (of form $[a_1, b_1] \times ... \times [a_1,b_2] \subset \R^n$).

\section{An SOS Optimization Problem For ROA Approximation} \label{sec: SOS opt}
For a given ODE~\eqref{eqn: ODE} we next propose a sequence of convex Sum-of-Squares (SOS) optimization problems, indexed by $d \in \N$. We show that the sequence of solutions, $\{P_d\}_{d\in \N}$, yields a sequence of sublevel sets which are contained inside the region of attraction of the ODE and which converge to the region of attraction of the ODE with respect to the volume metric as $d \to \infty$.

For given $f \in \mcl P(\R^n, \R^n)$, $\lambda>0$, $\beta \in \N$, $R>0$ and integration region $\Lambda \subset \R^n$ consider the following sequence of SOS optimization problems indexed by $d \in \N$:
\begin{align} \label{opt: SOS for ROA}
&P_d \in  \arg  \min_{J \in \mcl P_d(\R^n, \R)}  c^T \alpha \\ \nonumber
%\int_\Lambda J(x) dx\\ \nonumber
& J(x)=c^T Z_d(x), \\ \nonumber
&  k_1,k_2, s \in \sum_{SOS}^d \text{ and } p \in \mcl P_d(\R^n,\R) \\ \nonumber
&  J(0) \ge 0, \\ \nonumber
& k_1(x)= -\nabla J^T(x)f(x) - \lambda (1-J(x))||x||_2^{2 \beta} - s(x)(R^2 - ||x||_2^2), \\ \nonumber
& k_2(x)= (J(x)-1) - p(x)(R^2 - ||x||_2^2),
\end{align}
where $\alpha_i=\int_{\Lambda }  Z_{d,i}(x) dx $, recalling $Z_d: \R^n  \to \R^{ \mcl N_d}$ is the vector of monomials of degree $d \in \N$ and $ \mcl N_d= { d+ n \choose d}$.

We will show next, in Cor.~\ref{cor: SOS produces inner approx}, that the family of SOS optimization problems given in Eqn.~\eqref{opt: SOS for ROA} yields an inner approximation of $ROA_f$ for each $d \in \N$ (an approximation certifiably contained inside of $ROA_f$).
\begin{cor} \label{cor: SOS produces inner approx}
	Consider $f \in \mcl P (\R^n, \R)$ , $\lambda>0$, $\beta \in \N$, $R>0$ and $\Lambda \subset \R^n$. Suppose $ROA_f \subseteq B_R(0)$ and there exists $\eta>0$ such that $B_\eta(0)$ is an exponentially stable set. Then we have that
	\begin{align} \label{eqn: SOS gives subset of ROA}
	\{x \in B_R(0) :P_d(x) < 1\} \subseteq ROA_f \text{ for all } d \in \N,
	\end{align}
	where $P_d$ is any solution to the SOS Problem~\eqref{opt: SOS for ROA} for $d \in \N$.
\end{cor}
\begin{proof}
	Suppose $P_d$ is any solution to the SOS Problem~\eqref{opt: SOS for ROA} for $d \in \N$. Then $P_d$ satisfies the constraints of the SOS Problem~\eqref{opt: SOS for ROA} and thus satisfies Eqns~\eqref{eqn: diss }, \eqref{eqn: diss 2 }, and \eqref{eqn: diss 3} for $\Omega=B_R(0)$. Therefore, $W_{\lambda,\beta}(x) \le P_d(x)$ for all $x \in B_R(0)$ by Prop.~\ref{prop: diss ineq gives super LF}. Hence, it is clear that
	\begin{align}
\{x \in B_R(0) :P_d(x) < 1\} \subseteq \{x \in B_R(0) :W_{\lambda,\beta}(x) < 1 \}.
	\end{align}
	Moreover, Cor.~\ref{cor: ZLF characterizes ROA} shows $\{x \in B_R(0) :W_{\lambda,\beta}(x) < 1 \}=ROA_f$ and thus Eqn.~\eqref{eqn: SOS gives subset of ROA} holds.
\end{proof}
Cor.~\ref{cor: SOS produces inner approx} implies that solution maps initialized inside our $ROA_f$ approximation asymptotically coverage to the origin. That is for any $d \in \N$ and for all $y \in 	\{x \in B_R(0) :P_d(x) < 1\}$ we have that $\lim_{t \to \infty} ||\phi_f(y,t)||_2=0$, where $P_d$ is any solution to the SOS Problem~\eqref{opt: SOS for ROA} for $d \in \N$ (note this does not rule out the possibility that $\{x \in B_R(0) :P_d(x) < 1\}= \emptyset$).

Further to Cor.~\ref{cor: SOS produces inner approx}, we will next show, in Theorem~\ref{thm: SOS converges to ROA in V}, that for sufficiently large $\lambda>0$ and $\beta \in \N$ the sequence of SOS optimization problems given in Eqn.~\eqref{opt: SOS for ROA} yields a sequence of sets that tend to $ROA_f$ with respect to the volume metric as $d \to \infty$. We first define the volume metric. For sets $A,B \subset \R^n$, we denote the volume metric as $D_V(A,B)$, where
%\vspace{-0.15cm}
\begin{equation} \label{eqn: volume metric}
D_V(A,B):=\mu( (A/B) \cup (B/A) ).
\end{equation}
We note that $D_V$ is a metric (Defn.~\ref{def:metric}), as shown in Lem.~\ref{lem: Dv is metric} (found in Appendix~\ref{sec: appendix 2}).
\begin{thm} \label{thm: SOS converges to ROA in V}
	Consider $f \in \mcl P (\R^n, \R)$ and integration region $\Lambda \subset \R^n$. Suppose there exists $\theta,\eta,R>0$ such that $||D^\alpha f(x)||_2<\theta$ for all $x \in B_R(0)$ and $||\alpha||_1 \le 2$, $B_\eta(0)$ is an exponentially stable set (Defn.~\ref{defn: asym and exp stab}) of the ODE~\eqref{eqn: ODE}, and $ROA_f \subset B_R(0)$. Then if $ROA_f \subseteq \Lambda \subset B_R(0)$, $\lambda > \theta \eta^{-2 \beta}$ and $ \beta>\frac{\theta}{2\delta} +\frac{1}{2}$ we have that
\vspace{-0.2cm}
\begin{align} \label{convergence in DV for SOS}
\lim_{d \to \infty} D_V\bigg(ROA_f, \{x \in \Lambda : P_d(x) < 1 \} \bigg)  = 0,
\end{align}
where $P_d$ is any solution to Problem~\eqref{opt: SOS for ROA} for $d \in \N$.
\end{thm}
\begin{proof}
By Cor.~\ref{cor: ZLF characterizes ROA} we have that $ROA_f=\{x\in \R^n: W_{\lambda,\beta}(x)<1 \}$. Moreover, since $P_d$ satisfies the constraints of the SOS Problem~\eqref{opt: SOS for ROA} it follows that $P_d$ satisfies Eqns~\eqref{eqn: diss }, \eqref{eqn: diss 2 }, and \eqref{eqn: diss 3} for $\Omega=B_R(0)$. Therefore, $W_{\lambda,\beta}(x) \le P_d(x)$ for all $x \in B_R(0)$ by Prop.~\ref{prop: diss ineq gives super LF}. Thus, by Cor.~\ref{cor: close in L1 implies close in V norm} (found in Appendix~\ref{sec: appendix 2}) it follows that Eqn.~\eqref{convergence in DV for SOS} holds if $\lim_{d \to \infty} ||P_d - W_{\lambda,\beta}||_{L^1(\Lambda,\R)}=0$. To show $\lim_{d \to \infty} ||P_d - W_{\lambda,\beta}||_{L^1(\Lambda,\R)}=0$ we must show for all $\eps>0$ there exists $D \in \N$ such that
\begin{align} \label{L1 norm}
\int_{\Lambda} |P_d(x) - W_{\lambda,\beta}(x)| dx < \eps \text{ for all } d >D.
\end{align}
	
Now, let $\eps>0$. Then Theorem~\ref{thm: existence of feasible solution to SOS} shows there exists $J \in \mcl P(\R^n,\R)$ such that
	\begin{align} \label{p1}
& \sup_{x \in B_R(0)} | J(x) - W_{\lambda, \beta}(x) | < \frac{\eps}{\mu(\Lambda) +1},\\ \nonumber
& \nabla J(x)^Tf(x) <- \lambda (1-J(x))||x||_2^{2 \beta}   \text{ for all } x \in B_R(0), \\ \nonumber
& J(x) >1  \text{ for all } x \in \partial B_R(0) \text{ and } J(0) > 0.
\end{align}

Since $B_R(0)=\{x\in \R^n: R^2 - ||x||_2^2 \ge 0 \}$ and $\partial B_R(0)=\{x\in \R^n: R^2 - ||x||_2^2 \ge 0,  ||x||_2^2 - R^2 \ge 0  \}$ we have that by Putinar's Positivstellesatz (Theorem~\ref{thm: Psatz} given in Appendix~\ref{sec: appendix}) there exists $s_i \in \sum_{SOS}$ for $i \in \{1,...,5\}$ such that
\begin{align} \nonumber
& -\nabla J(x)^Tf(x) - \lambda (1-J(x))||x||_2^{2 \beta} \hspace{-0.15cm}   -s_1(x)(R^2 \hspace{-0.1cm} - ||x||_2^2)=s_2(x) , \\ \label{psatz1}
& \hspace{5cm} \text{for all } x \in \R^n.\\ \nonumber
& J(x) -1 - (s_3(x) -s_4(x)) (R^2 - ||x||_2^2)= s_5(x),\\ \label{psatz2}
& \hspace{5cm} \text{for all } x \in \R^n.
\end{align}

Let $D:=\max\{ \max_{i=1,..,5}\{deg(s_i)\}, deg(J) \}$. Then from Eqns~\eqref{psatz1} and \eqref{psatz2} and since $J(0)>0$ it follows that $J$ is feasible to the SOS Problem~\eqref{opt: SOS for ROA} for any $d>D$. Since, $P_d$ is the optimal solution to the SOS Problem~\eqref{opt: SOS for ROA} it follows that the objective function of the SOS Problem~\eqref{opt: SOS for ROA} evaluated at $P_d$ is less than or equal to the objective function evaluated at the feasible solution $J$ for $d>D$. That is by writing $P_d$ and $J$ with respect to the monomial vector, $P_d(x)=c_d^T Z_d(x)$ and $J(x)= b^T Z_{deg(J)}(x)$, we have that
\begin{align} \label{p2}
\int_{\Lambda} P_d(x) dx = c_d^T \alpha \le  b^T \gamma = \int_{\Lambda} J(x) dx \text{ for all } d>D,
\end{align}
where $\alpha_i=\int_{\Lambda }  Z_{d,i}(x) dx $, and $\gamma_i=\int_{\Lambda }  Z_{deg(J),i}(x) dx $.

We now show Eqn.~\eqref{L1 norm}. Using the fact $W_{\lambda,\beta}(x) \le P_d(x)$ for all $x \in \Lambda$ together with Eqns~\eqref{p1} and \eqref{p2} we get that,
\begin{align*}
& \int_{\Lambda} |P_d(x) - W_{\lambda,\beta}(x)|  dx  = \int_{\Lambda} P_d(x) dx - \int_{\Lambda} W_{\lambda,\beta}(x) dx\\
& \le \int_{\Lambda} J(x) dx - \int_{\Lambda} W_{\lambda,\beta}(x) dx \\
& \le \mu(\Lambda) \sup_{x \in \Lambda}\{|J(x) - W_{\lambda,\beta}(x) |\} < \eps \text{ for all } d>D.
\end{align*}
Hence by Cor.~\ref{cor: close in L1 implies close in V norm} (found in Appendix~\ref{sec: appendix 2}) it follows that Eqn.~\eqref{convergence in DV for SOS} holds.
\end{proof}
%In order to prove Theorem~\ref{thm: SOS converges to ROA in V} we reformulate the ``geometric" problem of computing estimations of the region of attraction to an ``algebriac" problem of computing a functional approximation of a certain converse Lyapunov function. In order to do this we must construct a converse Lyapunov function that characterizes the region of attraction. Moreover, this converse Lyapunov function must also be reasonably ``smooth" so it can be approximated by polynomials.

\section{Numerical Examples} \label{sec: numerical examples}
We now present several numerical examples that demonstrate that by solving the SOS problem, given in Eqn.~\eqref{opt: SOS for ROA}, we are able to approximate the region of attraction of a nonlinear system. Note that for numerical implementation it is best to choose $\lambda>0$ as small as possible. This is because the Lipschitz constant (given in Eqn.~\eqref{eqn: lip constant of ZLF}) of $W_{\lambda,\beta}$ (given in Eqn.~\eqref{fun: MLF}) grows as $\lambda>0$ increases.  For these numerical examples, we solve Opt.~\eqref{opt: SOS for ROA} using SOSTOOLS~\cite{sostools} to reformulate the problem as a Semi-Definite Programming (SDP) problem that is then solved by Sedumi~\cite{sturm1999using}.
		\begin{figure} 	
	\flushleft
	\includegraphics[scale=0.6]{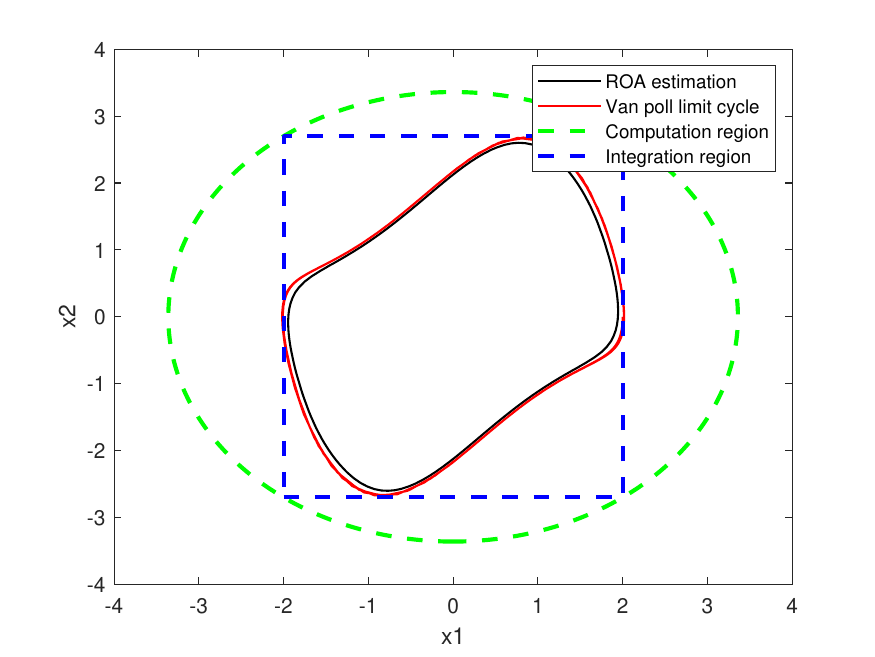}
		\vspace{-20pt}
	\caption{Graph showing an estimation of the region of attraction of the Van der Pol oscillator (Example~\ref{ex: van pol}) found by solving the SOS Problem~\eqref{opt: SOS for ROA}. The black line is the $1$-sublevel set of a solution to the SOS Problem~\eqref{opt: SOS for ROA}. The red line is the boundary of the region of attraction found by simulating a reverse time trajectory using Matlab's $\texttt{ODE45}$ function. The dotted blue line is the integration region, $\Lambda=[-2,2] \times [-2.7,2.7]$. The dotted green line is the computation region, $B_R(0)$ where $R=3.36$.   } \label{fig:perfect_d_6_van}
	\vspace{-10pt}
\end{figure}
\begin{ex} \label{ex: van pol}
Consider the Van der Pol oscillator defined by the ODE:
\begin{align} \label{eqn: van der pol ode}
\dot{x}_1(t) & = -x_2(t)\\ \nonumber
\dot{x}_2(t) & = x_1(t) - x_2(t)(1- x_1^2(t)).
\end{align}
In Fig.\ref{fig:perfect_d_6_van} we have plotted our estimation of the region of attraction of the ODE~\eqref{eqn: van der pol ode}. Our estimation is given by the $1$-sublevel set of the solution to the SOS optimization problem given in Eqn.~\eqref{opt: SOS for ROA} for $d=12$, $\lambda = 0.05$, $\beta=2$, $R=\sqrt{2^2 + 2.7^2} \approxeq 3.36$, $\Lambda=[-2,2] \times [-2.7,2.7]$, and $f=[-x_2 , x_1 + x_2 (x_1^2)]^T$.
\end{ex}
		\begin{figure} 	
	\flushleft
	\includegraphics[scale=0.6]{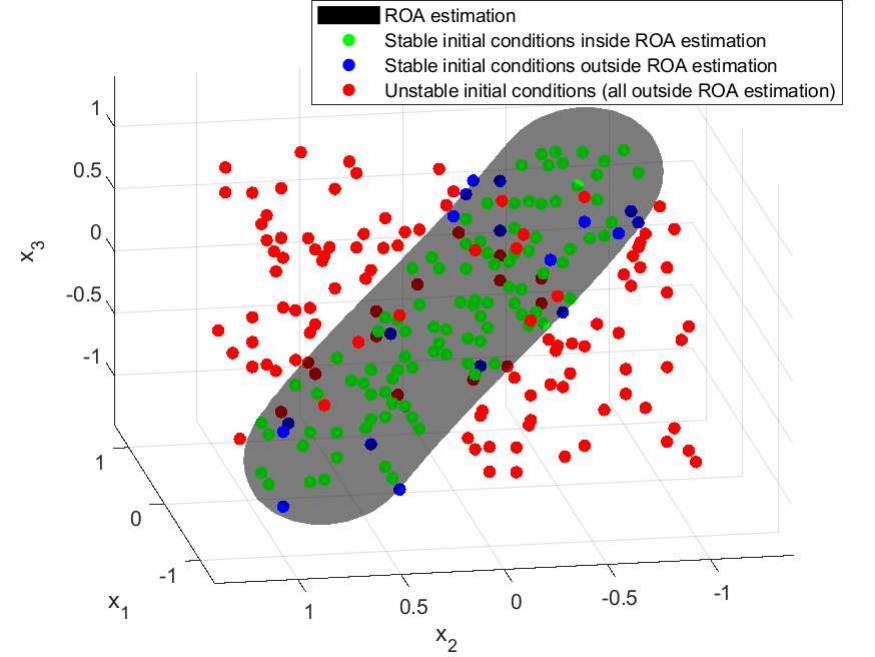}
	\vspace{-20pt}
	\caption{Graph showing an estimation of the region of attraction of servomechanism with multiplicative feedback control (Example~\ref{ex: servomechanism}). The estimation of the region of attraction is given by the transparent black sublevel set that is the $1$-sublevel set of a solution to the SOS Problem~\eqref{opt: SOS for ROA}. The scattered points are randomly generated initial conditions with associated trajectory (found using Matlab's \texttt{ODE45} function) that tends towards the origin (blue and green points) or away from the origin (red points). } \label{fig:3D}
	\vspace{-10pt}
\end{figure}
\begin{ex} \label{ex: servomechanism} Consider the third order servomechanism with multiplicative feedback control found in~\cite{ku1958stability} given by the following ODE:
	\begin{align} \label{ODE: thrid order}
T \frac{d^3 y}{dt^3} + \frac{d^2 y}{dt^2} + K_2(1-K_3 y^2)\frac{d y}{dt} + K_1 y=0,
	\end{align}
	where $T \in \R$ is a time constant and $K_1,K_2,K_2 \in \R$ are gain constants. We consider the case $T=K_2=1$ and $K_1=K_3=1$. The ODE~\eqref{ODE: thrid order} can be represented in the form $\dot{x}(t)=f(x(t))$ with
	\begin{align} \label{vector field: 3D}
	f(x)=[x_2,x_3, (1/T)(-x_3-K_2(1-K_3 x_1^2)x_2 - K_1 x_1)]^T.
	\end{align}
 Through numerical experiments using Matlab's \texttt{ODE45} function it was found that the ODE with vector field given in Eqn.~\eqref{vector field: 3D} appeared to have unbounded region of attraction. Therefore, for this system, Theorem~\ref{thm: SOS converges to ROA in V} does not show that the sequence of sublevel sets to the solution to the SOS problems given in Eqn.~\eqref{opt: SOS for ROA} converges to the region of attraction as $d \to \infty$. Nevertheless, in Fig.~\ref{fig:3D} we have plotted the $1$-sublevel set of the solution to the SOS optimization problem given in Eqn.~\eqref{opt: SOS for ROA} for $d=10$, $\lambda =0.5 $, $\beta=2$, $R=\sqrt{3}$, $\Lambda=[-1,1]^3$ and $f$ given in Eqn.~\eqref{vector field: 3D}. Fig.~\ref{fig:3D} indicates that even for systems with unbounded regions of attraction, our proposed SOS algorithm can provide arbitrarily good inner estimations of $ROA_f \cap \Lambda$, where $\Lambda \subset \R^n$ is some compact set. Through Monte Carlo simulation the volume of $ROA_f \cap \Lambda$ was estimated to be 0.3372 whereas the volume of our ROA approximation was found to be 0.2806, an error of 0.0566.
\end{ex}

\section{Conclusion} \label{sec: conclusion}
For a given locally exponentially stable dynamical system, described by an ODE, we have proposed a family SOS optimization problems that yields a sequence of sublevel sets that converge to the region of attraction of the ODE with respect to the volume metric. In order to facilitate this result we proposed a new converse Lyapunov function that was shown to be globally Lipschitz continuous. We have provided several numerical examples of practical interest showing how our proposed family of SOS problems can provide arbitrarily good approximations of regions of attraction. In future work we aim extend this work to systems with weaker forms of stability and investigate systems with unbounded regions of attraction.

\section{Acknowledgments}
This work was supported by the National Science Foundation under grants No. 1931270.

\bibliographystyle{unsrt}
\bibliography{bib_Zublov_LF}

\section{Appendix A: Approximation of Lipschitz Converse Lyapunov Functions in Sobolev Space} \label{sec: appendix mollification}
In this section we introduce aspects of mollification and polynomial approximation theory used in our proof of Theorem~\ref{thm: existence of feasible solution to SOS}; that there exists a polynomial arbitrarily ``close" to the converse Lyapunov function $W_{\lambda, \beta}$ (given in Eqn.~\eqref{fun: ZLF}) and is also a feasible solution to some $d \in \N$ instantiation of the family of SOS optimization problems given in Eqn.~\eqref{opt: SOS for ROA}. Theorem~\ref{thm: existence of feasible solution to SOS} is a key ingredient in the proof of Theorem~\ref{thm: SOS converges to ROA in V} (the main result of the paper).

\subsection{Approximating Lipschitz Functions by Infinitely Differentiable Functions}
For an overview of approximation by mollification we refer to~\cite{evans2010partial}.

\paragraph{Mollifiers}
The standard mollifier, $\eta \in C^\infty(\R^n , \R)$ is defined as

\vspace{-0.7cm}
\begin{align} \label{fun: mollifier}
\eta(x):=\begin{cases}
C \exp\left(\frac{1}{||x||_2^2 -1}\right) \quad \text{when } ||x||_2<1,\\
0 \quad \text{when } ||x||_2 \ge 1,
\end{cases}
\end{align}
where $C>0$ is chosen such that $\int_{\R^n } \eta(x) dx  =1$.

For $\sigma>0$ we denote the scaled standard mollifier by $\eta_\sigma \in C^\infty(\R^n , \R)$ such that

\vspace{-0.55cm}
\begin{equation*}
\eta_\sigma(x):= \frac{1}{\sigma^{n}} \eta\left(\frac{x}{\sigma} \right).
\end{equation*} Note, clearly $\eta_\sigma(x)=0$ for all $x \notin B_\sigma(0)$.

\paragraph{Mollification of a Function (Smooth Approximation)}
Recall from Section~\ref{sec: notation} that for open sets $\Omega \subset \R^n$ and $\sigma>0$ we denote $<\Omega>_\sigma:=\{x \in \Omega : B_\sigma(x) \subset \Omega  \}$. Now, for each $\sigma>0$ and function $V \in L^1(\Omega, \R)$ we denote the \textit{$\sigma$-mollification} of $V$ by $[V]_\sigma: <\Omega >_\sigma \to \R$, where
\vspace{-0.1cm}
\begin{align} \label{eqn: mollification}
[V]_\sigma(x) & := \int_{\R^n } \eta_\sigma(x-z)V(z)dz=\int_{B_\sigma(0)} \eta_\sigma(z) V(x-z) dz   .
\end{align}
To calculate the derivative of a mollification we next introduce the concept of weak derivatives.
\begin{defn} \label{def: weak deriv}
	For $\Omega \subset \R^n$ and $F \in L^1(\Omega, \R)$ we say any $H \in L^1(\Omega, \R)$ is the weak $i \in \{1,..,n\}$-partial derivative of $F$ if
	
	\vspace{-0.55cm}
	\begin{align*}
	\int_{\Omega} F(x) \frac{\partial}{\partial x_i} \alpha(x) dx = - \int_{\Omega} H(x) \alpha(x) dx,  \text{ for } \hspace{-0.05cm} \alpha \in C^\infty(\R^n,\R).
	\end{align*}
\end{defn}
Weak derivatives are ``essentially unique". That is if $H_1$ and $H_2$ are both weak derivatives of a function $F$ then the set of points where $H_1(x) \ne H_2(x)$ has measure zero. If a function is differentiable then its weak derivative is equal to its derivative in the ``classical" sense. We will use the same notation for the derivative in the ``classical" sense and in the weak sense.

%We now define the space of functions that have weak derivatives, called the Sobolev space. For $k \in \N$ and $1 \le p \le \infty$ we denote
%\begin{align*}
%W^{k,p}(\Omega,\R):=\{u\in L^p(\Omega,\R): D^\alpha u \in L^p(\Omega,\R) \forall |\alpha| \le k \}.
%\end{align*}
%The Sobolev norm for $u \in W^{k,p}(\Omega,\R)$ is defined as
%\begin{align*}
%||u||_{W^{k,p}(\Omega, \R)}:= \begin{cases}
%\left( \sum_{|\alpha| \le k} \int_\Omega (D^\alpha u(x))^p dx \right)^{\frac{1}{p}} \text{ if } 1 \le p < \infty\\
%\sum_{|\alpha| \le k} \esssup_{ x \in \Omega  } \{|D^\alpha u(x) |\} \text{ if } p= \infty.
%\end{cases}
%\end{align*}
%In the case $k=0$ the $W^{0,p}(\Omega,\R)$-norm is the $L^p(\Omega,\R)$-norm. Thus we use the notation $|| \cdot ||_{L^p(\Omega,\R)} :=|| \cdot ||_{W^{0,p}(\Omega,\R)} $.

In the next proposition we state some useful properties about Sobolev spaces and mollifications taken from \cite{evans2010partial}.

\begin{prop}[\cite{evans2010partial}] \label{prop:mollification}
	For $1 \le p < \infty$ and $k \in \N$ we consider $V \in W^{k,p}(E, \R)$, where $E\subset \R^{n}$ is an open bounded set, and its $\sigma$-mollification $[V]_\sigma$. Recalling from Section \ref{sec: notation} that for an open set $\Omega \subset \R^n$ and $\sigma>0$ we denote $<\Omega>_\sigma:=\{x \in \Omega: B(x,\sigma) \subset \Omega \}$, the following holds:
	\begin{enumerate}
		%	\item For all $1 \le  p < \infty$ we have $W^{1,p}(\Omega \times (0,T), \R) \subset W^{1, \infty}(\Omega \times (0,T), \R)$.
		\item For all $\sigma>0$ we have $[V]_\sigma \in C^\infty( <E>_\sigma, \R )$.
		%We need the approximation to be continuous over the closure of the set because we have a boundary condition at T to satisfy.
		\item For all $\sigma>0$ we have $\nabla_x [V]_\sigma(x)= [ \nabla_x  V]_\sigma(x)$ for $x \in {<E>_\sigma}$, where   $\nabla_x V$ is a weak derivatives.
		\item  If $V \in C(E,\R)$ then for any compact set $K \subset E$ we have $\lim_{\sigma \to 0}\sup_{(x) \in K}| V(x)- [V]_\sigma(x)|=0$.
%		\textcolor{red}{\item (Meyers-Serrin Local Approximation) For any compact set $K \subset E$ we have $\lim_{\sigma \to 0}\norm{ [V]_\sigma-V}_{W^{k,p}(K,\R)}=0$.}
	\end{enumerate}
\end{prop}
\subsection{Weighted Polynomial Approximation in Sobolev Space}
We next state a result that can be thought of as a generalization of the Weierstrass approximation theorem. It proves there exists a polynomial that can approximate a sufficiently smooth function arbitrarily well with respect to the $W^{1, \infty}$ norm weighted by a function of form $w(x)= 1/||x||_2^{2 \beta}$. This result was first presented in the case of $\beta=1$ in~\cite{peet2009exponentially} and then later extended the to general case of $\beta \in \N$ in~\cite{leth2017existence}.
\begin{thm}[Weighted Polynomial Approximation \cite{leth2017existence}] \label{thm:Nachbin}
{Let $E \subset \R^n$ be an open set, $\beta \in \N$ and $V \in C^{2 \beta +2}(\R^n, \R)$. For any compact set $K \subseteq E$ and $\eps>0$ there exists  $g \in \mcl P(\R^n, \R)$ such that}
	\begin{align*}
&|V(x) -  g(x)| < \eps ||x||_2^{2 \beta} \text{ for all } x \in K,\\
&||\nabla V(x) - \nabla g(x)||_2 < \eps ||x||_2^{2 \beta} \text{ for all } x \in K.
	\end{align*}
\end{thm}
\subsection{Approximation of Lyapunov Functions}
In this section we show in Theorem~\ref{thm: existence of feasible solution to SOS} that there exists a polynomial arbitrarily ``close" to the converse Lyapunov function $W_{\lambda, \beta}$ (given in Eqn.~\eqref{fun: ZLF}) and also a feasible solution to some $d \in \N$ instantiation of the family of SOS optimization problems given in Eqn.~\eqref{opt: SOS for ROA}. We take the following steps:
\begin{enumerate}[label=(\Alph*)]
	\item In Lemma~\ref{lem: mollification of ZLF} we take the mollification of $W_{\lambda, \beta}$ to show there exists an infinitely differentiable function that satisfies Eqns~\eqref{p 1}, \eqref{p 2} and \eqref{p 3}.
	\item In Prop.~\ref{prop: parition of unity ZLF} we use Lemma~\ref{lem: mollification of ZLF} together with partitions of unity (Theorem~\ref{thm: partition of unity}) to show there exists an infinitely differentiable function that satisfies Eqns~\eqref{p 12}, \eqref{p 22} and \eqref{p 32}.
	\item In Theorem~\ref{thm: existence of feasible solution to SOS} we use Prop.~\ref{prop: parition of unity ZLF} together with the polynomial approximation results in Theorem~\ref{thm:Nachbin} to show there exists a polynomial function that satisfies Eqns~\eqref{p 13}, \eqref{p 23} and \eqref{p 33}.
\end{enumerate}
\begin{lem} \label{lem: mollification of ZLF}
	Consider $f \in C^2(\R^n, \R)$ and $W$ as in Eqn.~\eqref{fun: ZLF}. Suppose there exists $\theta,\eta,R>0$ such that $||D^\alpha f(x)||_2<\theta$ for all $x \in B_R(x)$ and $||\alpha||_1 \le 2$, $B_\eta(0)$ is an exponentially stable set (Defn.~\ref{defn: asym and exp stab}) of the ODE~\eqref{eqn: ODE}, and $ROA_f \subset B_R(0)$. If $\lambda > \theta \eta^{-2 \beta}$ and $ \beta>\frac{\theta}{2\delta} +\frac{1}{2}$ then for any $\eps>0$ and $R_1>R$ there exists $J \in C^\infty (B_{R_1}(0), \R)$ such that
	\begin{align} \label{p 1}
	& \sup_{x \in B_{R_1}(0)} | J(x) - W_{\lambda, \beta}(x) | < \eps,\\ \label{p 2}
	& \nabla J(x)^Tf(x) <- \lambda (1-J(x))||x||_2^{2 \beta} + \eps  \text{ for all } x \in B_{R_1}(0), \\ \label{p 3}
	& J(x) =1  \text{ for all } x \in \partial B_R(0) \text{ and } J(0) \ge 0.
	\end{align}
\end{lem}
\begin{proof}
 Let $\eps>0$ and $R_2>R_1>R$. Since $W_{\lambda, \beta} \in Lip(\R^n, \R)$ (by Prop.~\ref{prop: ZLF is lip}) we know that by Theorem~\ref{thm: Rademacher theorem} that $W_{\lambda, \beta} \in W^{1, \infty}(\R^n, \R)$.

  For $\sigma>0$ let us denote the $\sigma$-mollification of $W_{\lambda, \beta}$ by $J_\sigma(x):=[W_{\lambda, \beta}]_\sigma(x)$. We note that the domain of $W_{\lambda, \beta}$ is $\R^n$. However, for mollification purposes we consider $W_{\lambda, \beta}$ over the restricted domain $B_{R_2}(0) \subset \R^n$.

  Let $\sigma_1:=\frac{R_2 - R_1}{2}$. It is clear that $B_{R_1}(0) \subset <B_{R_2}(0)>_\sigma$ for all $0<\sigma<\sigma_1$. Therefore, by Prop.~\ref{prop:mollification} we have that \linebreak $J_\sigma \in C^\infty(<B_{R_2}(0)>_\sigma, \R) \subset C^\infty(B_{R_1}(0), \R)$ for all $0<\sigma<\sigma_1$.

  We will now show there exists $\sigma>0$ such that Eqns~\eqref{p 1}, \eqref{p 2} and \eqref{p 3} hold.
	
	First we show Eqn.~\eqref{p 1} holds. By Prop.~\ref{prop:mollification} we know that there exists $\sigma_2>0$ such that for all $0<\sigma<\sigma_2$ we have that
	\begin{align*}
	\sup_{x \in B_{R_1}(0)} | J_\sigma (x) - W_{\lambda, \beta}(x) | < \eps.
	\end{align*}
	
	We now show Eqn.~\eqref{p 2} holds. Let us denote $r(x):=||x||_2^{2 \beta}$. It is clear using the triangle inequality and the fact that $||x-z||<2 R_1$ for all $x,z \in B_{R_1}(0)$ that
	\begin{align} \nonumber
	& r(x)-r(x-z)= (||x||_2-||x-z||_2)\sum_{k=0}^{2 \beta -1} ||x||_2^{2 \beta -1 -k} ||x-z||_2^{k} \\ \label{r ineq}
	& \le \left({R_1}^{2 \beta -1} \sum_{k=0}^{2 \beta -1} 2^k \right) ||z||_2 \text{ for all } x,z \in B_{R_1}(0).
	\end{align}
	  Let $\sigma_3:=\frac{\eps}{K L_f + \lambda\left( {R_1}^{2 \beta -1} \sum_{k=0}^{2 \beta -1} 2^k \right) }$ where $K$ (given in Eqn.~\eqref{eqn: lip constant of ZLF}) is the Lipschitz constants of $W_{\lambda, \beta}$ and $L_f$ is the Lipschitz constants of $f$. For $0< \sigma< \sigma_3$, using Prop.~\ref{prop:mollification} and the fact $W_{\lambda, \beta}$ satisfies Eqn.~\eqref{PDE: ZLF}, we have that \begin{align} \label{mollification diss ineq}
	& \nabla J_\sigma(x)^Tf(x) + \lambda (1-J_\sigma(x))||x||_2^{2 \beta}\\  \nonumber
	& =\nabla [W_{\lambda, \beta}]_\sigma(x)^Tf(x) + \lambda (1-[W_{\lambda, \beta}]_\sigma(x))r(x)\\ \nonumber
	& =[\nabla W_{\lambda, \beta}]_\sigma (x)^Tf(x) + \lambda (1-[W_{\lambda, \beta}]_\sigma(x))r(x)\\ \nonumber
	& = ([\nabla W_{\lambda, \beta}^T f ]_\sigma (x) + \lambda[r]_\sigma(x) - \lambda [W_{\lambda, \beta} r]_\sigma(x)) \\ \nonumber
	& \qquad + [\nabla W_{\lambda, \beta}]_\sigma (x)^Tf(x) -[\nabla W_{\lambda, \beta}^T f ]_\sigma (x) + \lambda r(x) - \lambda [r]_\sigma(x) \\ \nonumber
	& \qquad + \lambda [W_{\lambda, \beta} r]_\sigma(x) -\lambda [W_{\lambda, \beta}]_\sigma(x) r(x)\\ \nonumber
	& = [\nabla W_{\lambda, \beta}^T f + \lambda (1-W_{\lambda, \beta})r]_\sigma(x) \\ \nonumber
	& \qquad + [\nabla W_{\lambda, \beta}]_\sigma (x)^Tf(x) -[\nabla W_{\lambda, \beta}^T f ]_\sigma (x) \\ \nonumber
	& \qquad + \lambda(1- [W_{\lambda, \beta}]_\sigma) r(x) - \lambda[(1- W_{\lambda, \beta})r]_\sigma(x)\\  \nonumber
	& =  \int_{B_\sigma(0)} \eta_{\sigma}(z) \nabla W_{\lambda, \beta}(x-z)^T( f(x)- f(x-z) ) dz\\ \nonumber
	& \qquad + \lambda \int_{B_\sigma(0)} \eta_{\sigma}(z)  (1-W_{\lambda, \beta}(x-z))(  r(x) - r(x-z) ) dz \\ \nonumber
	& \le \esssup_{x \in \R^n}\{||\nabla W_{\lambda, \beta}(x)||_2\} \int_{B_\sigma(0)} \eta_{\sigma}(z) || f(x)- f(x-z)||_2 dz  \\ \nonumber
	& \qquad + \lambda  \int_{B_\sigma(0)} \eta_{\sigma}(z)  | r(x) - r(x-z) | dz \\  \nonumber
	& \le \left( K L_f  + \lambda  {R_1}^{2 \beta -1} \sum_{k=0}^{2 \beta -1} 2^k \right) \int_{B_\sigma(0)} \eta_{\sigma}(z)  ||z||_2 dz   \\  \nonumber
	& \le \left( K L_f + \lambda {R_1}^{2 \beta -1} \sum_{k=0}^{2 \beta -1} 2^k \right) \sigma < \eps \text{ for all } x \in B_{R_1}(0).
	\end{align}
	Where the first inequality in Eqn.~\eqref{mollification diss ineq} follows by the Cauchy Swartz inequality and the second inequality follows by the fact $\esssup_{x \in \R^n}\{||\nabla W_{\lambda, \beta}(x)||_2\} \le K$ (By Rademacher's theorem stated in Theorem~\ref{thm: Rademacher theorem}) and Eqn.~\eqref{r ineq}.
	
	We now show Eqn.~\eqref{p 3} holds. Since $ROA_f \subset B_R(0)$ and $ROA_f$ is an open set (by Lemma~\ref{lem: ROA open}) it follows that for all $x \in \partial B_R(0)$ we have $x \notin ROA_f$ and thus $W_{\lambda,\beta}(x)= 1$ for all $x \in \partial B_R(0)$. Now, there exists a sufficiently small $\sigma_4>0$ such that $B_{\sigma_4}(x) \cap ROA_f = \emptyset$ for all $x \in \partial B_R(0)$. Thus for $0<\sigma<\sigma_4$
	\begin{align*}
	J_\sigma(x)= \int_{B_\sigma(0)} \eta_\sigma(z) W(x-z) dz = \int_{B_\sigma(0)} \eta_\sigma(z)  dz =1,
	\end{align*}
	for all $x \in \partial B_R(0)$.
	
	Moreover, $\eta_{\sigma}(x) \ge 0$ and $W_{\lambda, \beta}(x) \ge 0$ for all $\sigma>0$ and $x \in \R^n$ so therefore $J_\sigma(x) \ge 0$ for all $\sigma>0$ and $x \in \R^n$. Thus $J_\sigma(0) \ge 0$ for all $\sigma>0$.
	
	In conclusion for $\sigma<\min\{\sigma_1, \sigma_2, \sigma_3, \sigma_4 \}$ we have that $J_\sigma$ satisfies Eqns~\eqref{p 1}, \eqref{p 2} and \eqref{p 3}.
\end{proof}
\begin{prop} \label{prop: parition of unity ZLF}
	Consider $f \in C^2(\R^n, \R)$ and $W$ as in Eqn.~\eqref{fun: ZLF}. Suppose there exists $\theta,\eta,R>0$ such that $||D^\alpha f(x)||_2<\theta$ for all $x \in B_R(x)$ and $||\alpha||_1 \le 2$, $B_\eta(0)$ is an exponentially stable set (Defn.~\ref{defn: asym and exp stab}) of the ODE~\eqref{eqn: ODE}, and $ROA_f \subset B_R(0)$. If $\lambda > \theta \eta^{-2 \beta}$ and $ \beta>\frac{\theta}{2\delta} +\frac{1}{2}$ then for any $\eps>0$ and $R_1>R$ there exists $J \in C^\infty(B_{R_1}(0), \R)$ such that
	\begin{align} \label{p 12}
	& \sup_{x \in B_{R_1}(0)} | J(x) - W_{\lambda, \beta}(x) | < \eps,\\ \label{p 22}
	& \nabla J(x)^Tf(x) \le - \lambda (1-J(x))||x||_2^{2 \beta} + \eps||x||_2^{2 \beta}  \text{ for } x \in B_{R_1}(0), \\ \label{p 32}
	& J(x) =1  \text{ for all } x \in \partial B_R(0) \text{ and } J(0) = 0.
	\end{align}
\end{prop}
\begin{proof}
Consider the sets $U_m= B_{R_1}(0)/(B_{1/m}(0))^{cl}$ for $m \in \N$. It is clear $\{U_m\}_{m \in \N}$ form an open cover (Defn.~\ref{defn: open cover}) of $B_{R_1}(0)/\{0\}$, that is $\cup_{m \in \N} U_m = B_{R_1}(0)/\{0\}$. By Theorem~\ref{thm: partition of unity} (found in Appendix~\ref{sec: appendix mollification}) there exists a partition of unity, we denote by $\{\psi_m \}_{m \in \N} \subset C^\infty(\R^n, \R)$, subordinate to the open cover $\{U_m\}_{m \in \N}$.

Let $\eps>0$. For each $m \in \N$ it was shown in Lemma~\ref{lem: mollification of ZLF} that there exists $J_m \in C^\infty(B_{R_1}(0), \R)$ such that
	\begin{align} \label{Jn uniform ineq}
& \sup_{x \in B_{R_1}(0)} | J_m(x) - W_{\lambda, \beta}(x) |\\ \nonumber
&  \qquad \qquad \qquad < \frac{\eps}{2^{m+1}(\sup_{x \in U_m}\{ |\nabla \psi_m(x)^T f(x)|\} +1)m^{2 \beta} },\\ \nonumber
& \nabla J_m(x)^Tf(x) <- \lambda (1-J_m(x))||x||_2^{2 \beta} + \frac{\eps}{2 m^{2 \beta}} \\ \label{Jn diss ineq}
& \qquad \qquad \qquad \qquad \qquad \qquad \qquad \text{ for all } x \in B_{R_1}(0), \\  \label{Jn BC}
& J_m(x) =1  \text{ for all } x \in \partial B_R(0) \text{ and } J_m(0) \ge 0.
\end{align}
Note, $\sup_{x \in U_m}\{ |\nabla \psi_m(x)^T f(x)|\}< \infty$ for each $m \in \N$ since $U_m$ is bounded and the function $\psi_m(x)^T f(x)$ is continuous in $x$.

We now consider the function $J(x):= \sum_{m=1}^\infty \psi_m(x) J_m(x)$. We first note that $J \in C^\infty( B_{R_1}(0), \R)$. This is due to the fact that $J_m \in C^\infty( B_{R_1}(0), \R)$ and $\psi_m \in C^\infty( \R^n, \R)$ for all $m \in \N$. Moreover, for any $x \in \R^n$ Theorem~\ref{thm: partition of unity} (found in Section~\ref{sec: appendix}) shows that there is an open set $S \subset \R^n$ containing $x \in \R^n$ such that only finitely many $\psi_m$'s are non-zero over $S$. Thus $J$ is a finite sum of $C^\infty( B_{R_1}(0), \R)$ functions over $S$ and thus differentiable at $x$. Since $x \in \R^n$ was arbitrarily chosen it follows $J \in C^\infty( B_{R_1}(0), \R)$.

We now show $J$ satisfies Eqn.~\eqref{p 12}. Using the fact $\sum_{m=1}^\infty \psi_m(x)=1$ for all $x \in B_{R_1}(0)/\{0\}$ and $\sum_{m=1}^\infty \psi_m(0)=0$ together with Eqn.~\eqref{Jn uniform ineq} we have that
\begin{align*}
&| J(x) - W_{\lambda, \beta}(x) |= \left|\sum_{m=1}^\infty \psi_m(x) J_m(x) - W_{\lambda, \beta}(x) \right|\\
&  \le \hspace{-0.1cm} \sum_{m=1}^\infty \hspace{-0.1cm} \psi_m(x) |J_m(x) \hspace{-0.05cm}- \hspace{-0.05cm} W_{\lambda, \beta}(x)| \hspace{-0.1cm} \le \hspace{-0.2cm} \sum_{m=1}^\infty  \frac{\psi_m(x) \eps}{2}< \eps \text{ for } x \in B_{R_1}(0).
\end{align*}

We now show $J$ satisfies Eqn.~\eqref{p 22}. Before doing so we note that $\sum_{m=1}^\infty \psi_m(x)=1$ for all $x \in B_{R_1}(0) /\{0\}$. Since only finitely many $\psi_m$'s are non-zero for each $x \in B_{R_1}(0) /\{0\}$ it follows $\sum_{m=1}^\infty \psi_m(x)$ is a finite sum of infinitely differentiable functions. Therefore, we can interchange the derivative and the summation to show $0=\frac{\partial}{\partial x_i} 1 = \frac{\partial}{\partial x_i} \sum_{m=1}^\infty \psi_m(x) = \sum_{m=1}^\infty \frac{\partial}{\partial x_i} \psi_m(x)$ for all $x \in B_{R_1}(0)/\{0\}$ and $i \in \{1,...,n\}$. Thus it follows $ \sum_{m=1}^\infty \nabla \psi_m(x) = [0,...,0]^T \in \R^n$ for all $x \in B_{R_1}(0)/\{0\}$. Hence,
\begin{align} \label{zero identity}
W_{\lambda \beta}(x) \sum_{m=1}^\infty  \nabla \psi_m(x)^Tf(x)=0 \text{ for all } x \in B_{R_1}(0)/\{0\}.
\end{align}

For $x \in B_{R_1}(0)/\{0\}$ let us denote $I_x:=\{m \in \N: x \in U_m  \}$. Note, $\{U_m\}_{m \in \N}$ forms an open cover for $B_{R_1}(0)/\{0\}$ so $I_x \ne \emptyset$ for all $x \in B_{R_1}(0)/\{0\}$.

It is clear that for $x \in B_{R_1}(0)/\{0\}$ and $m \in I_x$ that $x \in U_m= B_{R_1}(0)/B_{\frac{1}{m}}(0)$ and so $||x||_2 \ge \frac{1}{m}$ implying $\frac{1}{m^{2\beta}} \le ||x||_2^{2\beta}$. Therefore,
\begin{align} \label{Ix1}
\sup_{m \in I_x}\left\{\frac{1}{m^{2\beta}}\right\} \le ||x||_2^{2\beta} \text{ for all } x \in B_{R_1}(0)/\{0\}.
\end{align}
Moreover, for $x \in B_{R_1}(0)/\{0\}$ and $m \notin I_x$ we have that $x \notin U_m$. Thus, since $\{x \in \R^n: \psi_m(x) \ne 0\} \subset U_m$ (by Theorem~\ref{thm: partition of unity} found in Appendix~\ref{sec: appendix}) we have that
\begin{align} \label{Ix2}
\psi_m(x)=0 \text{ for all } x \in B_{R_1}(0)/\{0\} \text{ and }  m \notin I_x.
\end{align}

Now, using Eqns~\eqref{Jn uniform ineq}, \eqref{Jn diss ineq}, \eqref{zero identity}, \eqref{Ix1} and \eqref{Ix2}, and the fact $\sum_{m=1}^\infty \frac{1}{2^m}=1$ we have that,
\begin{align} \label{123}
& \nabla J(x)^Tf(x) + \lambda (1-J(x))||x||_2^{2 \beta} \\ \nonumber
& = \sum_{m=1}^\infty \psi_m(x) \left(\nabla J_m(x)^Tf(x) + \lambda (1-J_m(x))||x||_2^{2 \beta} \right) \\ \nonumber
& \qquad + \sum_{m=1}^\infty J_m(x) \nabla \psi_m(x)^Tf(x) - W_{\lambda \beta}(x) \sum_{m=1}^\infty  \nabla \psi_m(x)^Tf(x)\\ \nonumber
& = \sum_{m \in I_x} \psi_m(x) \left(\nabla J_m(x)^Tf(x) + \lambda (1-J_m(x))||x||_2^{2 \beta} \right) \\ \nonumber
& \qquad + \sum_{m \in I_x} (J_m(x) -W_{\lambda \beta}(x)) \nabla \psi_m(x)^Tf(x) \\ \nonumber
& \le \sum_{m \in I_x} \psi_m(x) \frac{\eps}{2 m^{2 \beta}} + \sum_{m \in I_x} \frac{\eps}{2^{m+1} m^{2 \beta}} \\ \nonumber
&\le \eps \sup_{m \in I_x} \left\{ \frac{1}{ m^{2 \beta}} \right\} \left( \frac{1}{2} \sum_{m \in I_x} \psi_m(x) + \frac{1}{2}\sum_{m \in I_x}\frac{1}{2^m}  \right)  \\ \nonumber
&  \le \eps \sup_{m \in I_x} \left\{ \frac{1}{m^{2 \beta}} \right\}  \le \eps ||x||_2^{2\beta} \text{ for all } x \in B_{R_1}(0)/\{0\}.
\end{align}

Eqn.~\eqref{123} shows $J$ satisfies Eqn.~\eqref{p 22} for $x \in B_{R_1}(0)/\{0\}$. We still need to show $J$ satisfies Eqn.~\eqref{p 22} for $x=0$. Let us denote the function $F(x):=\nabla J(x)^Tf(x) + \lambda (1-J(x))||x||_2^{2 \beta}$. To show $J$ satisfies Eqn.~\eqref{p 22} at $x=0$ we must show $F(0) \le 0$. We first note that $F \in C^2(B_{R_1}(0), \R)$ since $J \in C^\infty(B_{R_1}(0), \R)$, $f \in C^2(\R^n, \R)$ and $||x||_2^{2 \beta} \in C^2(\R^n, \R)$. Thus $F \in LocLip(\R^n, \R)$. Therefore,
\begin{align} \label{1234}
|F(0)-F(x)| \le L_F ||x||_2 \text{ for all } x \in B_{R_1}(0),
\end{align} where $L_F$ is the Lipschitz constant of $F$. Then, Eqn.~\eqref{123} together with Eqn.~\eqref{1234} implies that
\begin{align} \label{12345}
& F(0) \le L_F ||x||_2 + F(x) \le L_F ||x||_2 + \eps ||x||_2^{2 \beta}\\ \nonumber
&\qquad \qquad \qquad \qquad  \text{ for all } x \in B_{R_1}(0)/\{0\}.
\end{align}
Now, for contradiction suppose the negation of $F(0) \le 0$, that is there exists $a>0$ such that $F(0)\ge a$. Considering $x=\min\{\frac{a}{ 3(L_F+1) \sqrt{n}},\frac{1}{\sqrt{n}}(\frac{a}{ 3 \eps})^{1/\beta}, \frac{R_1}{\sqrt{n}}  \}[1,...,1]^T \in B_{R_1}(0)/\{0\} \subset \R^n$ and using Eqn.~\eqref{12345} we have that
\begin{align*}
a \le F(0) \le \frac{2}{3}a,
\end{align*}
providing a contradiction. Therefore, $F(0) \le 0$ and so $J$ satisfies Eqn.~\eqref{p 22} for all $x \in B_{R_1}(0)$.

We now show $J$ satisfies Eqn.~\eqref{p 32}. Let $x \in \partial B_R(0)$. By Eqn.~\eqref{Jn BC} we have that $J_m(x)=1$ for all $m \in \N$. Therefore, using the fact $\sum_{m=1}^\infty \psi_m(x) =1$ for all $x \in B_{R_1}(0)/\{0\}$ and $\partial B_{R}(0) \subset  B_{R_1}(0)/\{0\}$ since $R_1>R$, we have that
\begin{align*}
J(x)= \sum_{m=1}^\infty \psi_m(x) J_m(x) = \sum_{m=1}^\infty \psi_m(x) =1.
\end{align*}
Moreover, $0 \notin B_{R_1}(0)/\{0\}$ so $\psi_m(0)=0$ for all $m \in \N$. Hence $J(0)= \sum_{m=1}^\infty \psi_m(x) J_m(x)=0$.
\end{proof}

\begin{thm} \label{thm: existence of feasible solution to SOS}
	Consider $f \in C^2(\R^n, \R)$ and $W$ as in Eqn.~\eqref{fun: ZLF}. Suppose there exists $\theta,\eta,R>0$ such that $||D^\alpha f(x)||_2<\theta$ for all $x \in B_R(x)$ and $||\alpha||_1 \le 2$, $B_\eta(0)$ is an exponentially stable set (Defn.~\ref{defn: asym and exp stab}) of the ODE~\eqref{eqn: ODE}, and $ROA_f \subset B_R(0)$. If $\lambda > \theta \eta^{-2 \beta}$ and $ \beta>\frac{\theta}{2\delta} +\frac{1}{2}$ then for any $\eps>0$ there exists $P \in \mcl P(\R^n, \R)$ such that
	\begin{align} \label{p 13}
	& \sup_{x \in B_R(0)} | P(x) - W_{\lambda, \beta}(x) | < \eps,\\ \label{p 23}
	& \nabla P(x)^Tf(x) <- \lambda (1-P(x))||x||_2^{2 \beta}   \text{ for all } x \in B_R(0), \\ \label{p 33}
	& P(x) >1  \text{ for all } x \in \partial B_R(0) \text{ and } P(0) > 0.
	\end{align}
\end{thm}
\begin{proof} Let $\eps>0$ and $R_1>R$. By Prop.~\ref{prop: parition of unity ZLF} there exists $J \in C^\infty(B_{R_1}(0), \R)$ that satisfies
	\begin{align} \label{p12}
	& \sup_{x \in B_{R_1}(0)} | J(x) - W_{\lambda, \beta}(x) | < \frac{\eps}{a},\\ \label{p22}
	& \nabla J(x)^Tf(x) \le - \lambda (1-J(x))||x||_2^{2 \beta} + \frac{\eps}{a}||x||_2^{2 \beta}  \text{ for } x \in B_{R_1}(0), \\ \label{p32}
	& J(x) =1  \text{ for all } x \in \partial B_R(0) \text{ and } J(0) = 0,
	\end{align}
	where 
	\begin{align} \label{a}
a:= \max \left\{ 3,\frac{\sup_{x \in B_R(0)} ||f(x)||_2}{ \lambda R}  +  R^{-2 \beta} + \frac{1}{\lambda} +2  \right\}.
	\end{align}
	
Now, Theorem~\ref{thm:Nachbin}, found in Section~\ref{sec: appendix mollification}, shows there exists $\tilde{P} \in \mcl P(\R^n , \R)$ such that
	\begin{align} \label{J close to p}
	& |J(x)-\tilde{P}(x)| < \frac{\eps}{a R^{2 \beta}} ||x||_2^{2\beta} \text{ for all } x \in  (B_{R}(0))^{cl}, \\ \label{J nabla t close to P}
	& ||\nabla J(x)- \nabla \tilde{P}(x)||_2< \frac{\eps}{a R^{2 \beta}} ||x||_2^{2\beta} \text{ for all } x \in  (B_{R}(0))^{cl}.
	\end{align}

	Let $P(x):=\tilde{P}(x) + \frac{a-2}{a} \eps \in \mcl P(\R^n, \R)$. We will now show $P$ satisfies Eqns~\eqref{p 13}, \eqref{p 23} and \eqref{p 33}.
	
		We first show Eqn.~\eqref{p 13} holds. Using the triangle inequality along with Eqns~\eqref{p12} and \eqref{J close to p} we have that
		\begin{align*}
		& |P(x)-W_{\lambda,\beta}(x)| \le 	|\tilde{P}(x)-W_{\lambda,\beta}(x)| + \frac{a-2}{a} \eps\\
		& \le  |\tilde{P}(x)-J(x)| + |J(x)-W_{\lambda,\beta}(x)| + \frac{a-2}{a} \eps\\
		& \le \frac{\eps}{a} + \frac{\eps}{a} + \frac{a-2}{a} \eps = \eps.
		\end{align*}
	
	We now show Eqn.~\eqref{p 23} holds. Using Eqns~\eqref{p22}, \eqref{a}, \eqref{J close to p}, and \eqref{J nabla t close to P} we have that
	\begin{align*}
	& \nabla P(x)^Tf(x) + \lambda (1-P(x))||x||_2^{2 \beta}\\
	& \le \nabla \tilde{P}(x)^Tf(x) + \lambda (1-\tilde{P}(x))||x||_2^{2 \beta} - \lambda\frac{\eps(a-2)}{a} ||x||_2^{2 \beta}  \\
	& \qquad - \nabla J(x)^Tf(x) - \lambda (1-J(x))||x||_2^{2 \beta} + \frac{\eps}{a}||x||_2^{2 \beta} \\
	& = (\nabla \tilde{P}(x) -\nabla J(x))^T f(x) + \lambda (J(x)- \tilde{P}(x) ) ||x||_2^{2 \beta} \\
	&\qquad +  \frac{\eps}{a}\bigg(1 - \lambda(a-2) \bigg)||x||_2^{2 \beta} \\
	& \le || \nabla \tilde{P}(x) -\nabla J(x)||_2 ||f(x)||_2 + \lambda R^{-2 \beta} \frac{\eps}{a} ||x||_2^{2 \beta} \\
	& \qquad + \frac{\eps}{a}(1 - \lambda(a-2))||x||_2^{2 \beta} \\
	& \le \bigg( \frac{\sup_{x \in B_R(0)} ||f(x)||_2}{R}  + \lambda R^{-2 \beta} + 1 - \lambda(a-2) \bigg)\frac{\eps}{a} ||x||_2^{2 \beta}\\
	& \le 0.
	\end{align*}
	
	We now show Eqn.~\eqref{p 33} holds. From Eqn.~\eqref{J close to p} we have that $\tilde{P}(x) > J(x) - \frac{\eps}{a R^{2 \beta}}||x||_2^{2 \beta}$ for all $x \in (B_R(0))^{cl}$. Moreover, Eqn.~\eqref{p32} we have that $J(x)=1$ for all $x \in \partial B_R(0)$. Therefore $P(x)= \tilde{P}(x) + \frac{a-2}{a}\eps > 1 + \frac{a-2}{a}\eps - \frac{\eps}{a R^{2 \beta}}||x||_2^{2 \beta}>1 + \frac{a-3}{a}\eps>1 $. Also from Eqn.~\eqref{J close to p} we have that $\tilde{P}(0)=J(0)$. From Eqn.~\eqref{p32} we have that $J(0)=0$. Therefore $P(0)=\tilde{P}(0) + \frac{a-2}{a} \eps>0$.
\end{proof}
\section{Appendix B: sublevel set approximation} \label{sec: appendix 2}
This appendix is concerned with the volume metric ($D_V$ in Eqn.~\eqref{eqn: volume metric}). The sublevel approximation results presented in this appendix are required in the proof of Theorem~\ref{thm: SOS converges to ROA in V}. %In Cor.~\ref{cor: close in L1 implies close in V norm} we show that if $V(x) \le J_d(x)$ and $\lim_{d \to \infty} ||J_d - V||_{L^1}=0$, then for any $\gamma \in \R$ we have $\lim_{d \to \infty}	D_V (\{x \in \Lambda : V(x) < \gamma\}, \{x \in \Lambda : J_d(x) < \gamma\} ) =0$.
\begin{defn} \label{def:metric}
	$D: X \times X \to \R$ is a \textit{metric} if the following is satisfied for all $x,y \in X$,
	\vspace{-0.4cm}
	\setlength{\columnsep}{-0.75in}
	\begin{multicols}{2}
		\begin{itemize}
			\item $D(x,y) \ge 0$,
			\item $D(x,y)=0$ iff $x=y$,
			\item $D(x,y)=D(y,x)$,
			\item $D(x,z) \le D(x,y) + D(y,z)$.
		\end{itemize}
	\end{multicols}
	%	\begin{AutoMultiColItemize}
	%		\item $D(x,y) \ge 0$,
	%		\item $D(x,y)=0$ iff $x=y$,
	%		\item $D(x,y)=D(y,x)$,
	%		\item $D(x,z) \le D(x,y) + D(y,z)$.
	%	\end{AutoMultiColItemize}
	%	\begin{itemize}
	%		\item $D(x,y) \ge 0$,
	%		\item $D(x,y)=0$ iff $x=y$,
	%		\item $D(x,y)=D(y,x)$,
	%		\item $D(x,z) \le D(x,y) + D(y,z)$.
	%	\end{itemize}
\end{defn}
%We now prove $D_V$ is indeed a metric. \textcolor{red}{For unbounded sets is this still a metric?}
\begin{lem}[\cite{jones2019using}] \label{lem: Dv is metric}
	Consider the quotient space,
	{	\[
		X:= \mcl B \pmod {\{X \subset \R^n : X \ne \emptyset, \mu(X) =0 \}},
		\] } recalling $\mcl B:= \{B \subset \R^n: \mu(B)<\infty\}$ is the set of all bounded sets. Then $D_V: X \times X \to \R$, defined in Eqn.~\eqref{eqn: volume metric}, is a metric.
\end{lem}

%We next show that when $B \subseteq A$ the metric $D_V$ greatly simplifies.

\begin{lem}[\cite{jones2019using}] \label{lem: D_V is related to vol}
	If $A,B \in \mcl B$ and  $B \subseteq A$ then
	\begin{align*}
	D_V(A,B)& =\mu(A/B)= \mu(A)- \mu (B).
	\end{align*}
\end{lem}
\begin{prop}[\cite{jones2020polynomial}] \label{prop: close in L1 implies close in V norm}
	Consider a set $\Lambda \in \mcl B$, a function $V \in L^1(\Lambda, \R)$, and a family of functions $\{J_d \in L^1(\Lambda, \R): d \in \N\}$ that satisfies the following properties:
	\begin{enumerate}
		\item For any $ d \in \N$ we have $J_d(x) \le V(x)$ for all $x \in \Lambda$.
		\item $\lim_{d \to \infty} ||V -J_d||_{L^1(\Lambda, \R)} =0$.
	\end{enumerate}
	Then for all $\gamma \in \R$
	\begin{align} \label{sublevel sets close}
	\lim_{d \to \infty}	D_V \bigg(\{x \in \Lambda : V(x) \le \gamma\}, \{x \in \Lambda : J_d(x) \le \gamma\} \bigg) =0.
	\end{align}
\end{prop}

\begin{cor} \label{cor: close in L1 implies close in V norm}
Consider a set $\Lambda \in \mcl B$, a function $V \in L^1(\Lambda, \R)$, and a family of functions $\{J_d \in L^1(\Lambda, \R): d \in \N\}$ that satisfies the following properties:
\begin{enumerate}
	\item For any $ d \in \N$ we have $J_d(x) \ge V(x)$ for all $x \in \Lambda$.
	\item $\lim_{d \to \infty} ||V -J_d||_{L^1(\Lambda, \R)} =0$.
\end{enumerate}
Then for all $\gamma \in \R$
\begin{align} \label{strict sublevel sets close}
\lim_{d \to \infty}	D_V \bigg(\{x \in \Lambda : V(x) < \gamma\}, \{x \in \Lambda : J_d(x) < \gamma\} \bigg) =0.
\end{align}
\end{cor}
\begin{proof}
Consider some $\gamma \in \R$. Let us denote $\tilde{V}(x)=-V(x)$, $\tilde{J}_d(x)= - J_d(x)$ and $\gamma=-\tilde{\gamma}$. It follows that $\tilde{J}_d(x) \le \tilde{V}(x)$ for all $x \in \Lambda$ and $\lim_{d \to \infty} ||\tilde{V} -\tilde{J}_d||_{L^1(\Lambda, \R)} =0$. Therefore, by Prop.~\ref{prop: close in L1 implies close in V norm} we have that
\begin{align} \label{pp1}
\lim_{d \to \infty}	D_V \bigg(\{x \in \Lambda : \tilde{V}(x) \le \tilde{\gamma} \}, \{x \in \Lambda : \tilde{J}_d(x) \le \tilde{\gamma} \} \bigg) =0.
\end{align}

Now, $\Lambda=\{x \in \Lambda : V(x) < \gamma\} \cup \{x \in \Lambda : V(x) \ge \gamma\} = \{x \in \Lambda : V(x) < \gamma\} \cup \{x \in \Lambda : \tilde{V}(x) \le \tilde{\gamma} \}$. Therefore
\begin{align*}
\{x \in \Lambda : V(x) < \gamma\} = \Lambda / \{x \in \Lambda : \tilde{V}(x) \le \tilde{\gamma} \},
\end{align*}
and by a similar argument
\begin{align*}
\{x \in \Lambda : J_d(x) < \gamma\} = \Lambda / \{x \in \Lambda : \tilde{J}_d(x) \le \tilde{\gamma} \}.
\end{align*}
Thus, by Lem.~\ref{lem: D_V is related to vol} and since $\{x \in \Lambda : \tilde{J}_d(x) \le \tilde{\gamma} \} \subseteq \Lambda$, we have that
\begin{align} \label{pp2}
& D_V \bigg(\{x \in \Lambda : V(x) < \gamma\}, \{x \in \Lambda : J_d(x) < \gamma\} \bigg)\\ \nonumber
& = D_V \bigg( \Lambda/ \{x \in \Lambda : \tilde{V}(x) < \tilde{\gamma} \}, \Lambda/ \{x \in \Lambda : \tilde{J}_d(x) < \tilde{\gamma} \} \bigg) \\ \nonumber
& = D_V \bigg(\{x \in \Lambda : \tilde{V}(x) \le \tilde{\gamma} \}, \{x \in \Lambda : \tilde{J}_d(x) \le \tilde{\gamma} \} \bigg).
\end{align}
Now by Eqns~\eqref{pp1} and~\eqref{pp2} it follows that Eqn.~\eqref{strict sublevel sets close} holds.
\end{proof}

\section{Appendix C} \label{sec: appendix}
In this appendix we present several miscellaneous results required in various places throughout the paper and not previously found in any of the other appendices.
\begin{lem}[Exponential inequalities] \label{lem: exp ineq}
	The following inequalities hold
	\begin{align} \label{ineq:exp 1}
	\exp(-x) \le 1 \text{ for all } x \ge 0\\  \label{ineq:exp 2}
	x \exp(-x) \le 1 \text{ for all } x \in \R. \\ \label{ineq:exp 3}
	\exp(x) \ge 1 + x \text{ for all } x \in \R.
	\end{align}
\end{lem}

\begin{lem}[Gronwall's Inequality \cite{Hirch_2004}] \label{lem: gronwall}
	Consider scalars $a,b \in \R$ and functions $u, \beta \in C^1(I,\R)$. Suppose
	\begin{align*}
	\frac{d}{dt}u(t) \le \beta(t) u(t) \text{ for all } t \in (a,b).
	\end{align*}
	Then it follows that
	\begin{align*}
	u(t) \le u(a) \exp\left(\int_a^t \beta(s) ds\right) \text{ for all } t \in [a,b].
	\end{align*}
\end{lem}

\begin{thm}[Rademacher's Theorem \cite{maly1997fine} \cite{evans2010partial}] \label{thm: Rademacher theorem}
	If $\Omega \subset \R^n$ is an open subset and $V \in Lip(\Omega, \R)$, then $V$ is differentiable almost everywhere in $\Omega$ with point-wise derivative corresponding to the weak derivative almost everywhere; that is the set of points in $\Omega$ where $V$ is not differentiable has Lebesgue measure zero. Moreover,
	\vspace{-0.1cm}\begin{align*}
	\esssup_{x \in \Omega} \bigg|\frac{\partial}{\partial x_i}V(x) \bigg| \le L_V \text{ for all } 1 \le i \le n,
	\end{align*}
	where $L_V>0$ is the Lipschitz constant of $V$ and $\frac{\partial}{\partial x_i}V(x)$ is the weak derivative of $V$.
\end{thm}

\begin{thm}[Putinar's Positivstellesatz \cite{Putinar_1993}] \label{thm: Psatz}
	Consider the semialgebriac set $X = \{x \in \R^n: g_i(x) \ge 0 \text{ for } i=1,...,k\}$. Further suppose $\{x  \in \R^n : g_i(x) \ge 0 \}$ is compact for some $i \in \{1,..,k\}$. If the polynomial $f: \R^n \to \R$ satisfies $f(x)>0$ for all $x \in X$, then there exists SOS polynomials $\{s_i\}_{i \in \{1,..,m\}} \subset \sum_{SOS}$ such that,
	\vspace{-0.4cm}\begin{equation*}
	f - \sum_{i=1}^m s_ig_i \in \sum_{SOS}.
	\end{equation*}
\end{thm}
\begin{thm}[The Bolzano Weierstrass Theorem] \label{thm: Bolzano}
	Consider a sequence $\{x_n\}_{n \in \N} \subset \R^n$. Then the $\{x_n\}_{n \in \N}$ is a bounded sequence, that is there exists $M>0$ such that $x_n < M$ for all $n \in \N$, if and only if there exists a convergent subsequence $\{y_n\}_{n \in \N} \subset \{x_n\}_{n \in \N}$.
\end{thm}

%\begin{thm}[Leibniz integral rule] \label{thm: Leibniz}
%	Let $X\subset \R$ be open, and $\Omega$ a measure space. Suppose the function $f: X \times \Omega \to \R$ satisfies the following,
%	\begin{enumerate}
%		\item $f(x,\omega)$ is a Lebesgue-integrable function of $\omega$ for all $x \in X$.
%		\item For all $\omega \in \Omega$, the derivative $f_x$ exists for all $x \in X$.
%		\item There is an integrable function $g : \Omega \to \R$ such that $|f_x(x, \omega)| \le g(\omega)$ for all $x \in X$ and almost all $\omega \in \Omega$.
%	\end{enumerate}
%	Then, $\forall x \in X$ we have $ \frac{d}{dx} \int_{\Omega} f(x, \omega) d \omega = \int_\Omega f_x (x, \omega) d \omega$.
%\end{thm}
\begin{defn} \label{defn: open cover}
	Let $\Omega \subset \R^n$. We say $\{U_i \}_{i=1}^\infty$ is an open cover for $\Omega$ if $U_i \subset \R^n$ is an open set for each $i \in \N$ and $\Omega \subseteq \{U_i \}_{i=1}^\infty$.
\end{defn}

\begin{thm}[Existence of Partitions of Unity \cite{spivak2018calculus}] \label{thm: partition of unity}
	Let $U \subseteq \R^n$ and let $\{U_i\}_{i=1}^\infty$ be an {open cover} of $E$. Then there exists a collection of $C^\infty(\R^n, \R)$ functions, denoted by $\{\psi\}_{i=1}^\infty$, with the following properties:
	\begin{enumerate}
		\item For all $x \in U$ and $i \in \N$ we have $0 \le \psi_i(x) \le 1$.
		\item For all $x \in U$ there exists an open set $S \subseteq U$ containing $x$ such that all but finitely many $\psi_i$ are 0 on $S$.
		\item For each $x \in U$ we have $\sum_{i=1}^\infty \psi_i(x) =1$.
		\item For each $i \in \N$ we have $\{x \in U: \psi_i(x) \ne 0\} \subseteq U_i$.
	\end{enumerate}
\end{thm}

\end{document}